\documentclass[12pt]{amsart}

\usepackage{cite}

\usepackage{kotex}
\usepackage{lineno,hyperref}
\usepackage{epstopdf}
\usepackage{pifont}
\usepackage{dsfont}
\usepackage{tikz}
\usepackage{latexsym}
\usepackage{graphics}
\usepackage{graphicx}
\usepackage{float}
\usepackage[dvips]{xy}
\xyoption{all}
\usepackage{ifpdf}
\usepackage{multirow}
\usepackage{amssymb, amsthm, amsmath}
\usepackage{hhline}
\usepackage{centernot}
\usepackage{verbatim, comment, color}
\usepackage{enumitem}
\usepackage{setspace}
\usepackage{lipsum}
\usepackage{subcaption}
\usepackage{mathtools}

\usepackage[numbers,sort&compress]{natbib}

\newtheorem{theorem}{Theorem} 
\newtheorem{proposition}[theorem]{Proposition}
\newtheorem{lemma}[theorem]{Lemma}
\newtheorem{remark}[theorem]{Remark}
\newtheorem{corollary}[theorem]{Corollary}
\newtheorem{definition}[theorem]{Definition}
\newtheorem{example}[theorem]{Example}

\newtheorem{question}{Question}

\newcommand{\ba}{\begin{align}}
\newcommand{\ea}{\end{align}}  

\newcommand{\be}{\begin{equation}}
\newcommand{\ee}{\end{equation}}

\newcommand{\bea}{\begin{eqnarray}}
\newcommand{\eea}{\end{eqnarray}}
\newcommand{\barr}{\begin{array}}
\newcommand{\earr}{\end{array}}
\newcommand{\bn}{\begin{enumerate}}
\newcommand{\en}{\end{enumerate}}
\newcommand{\bi}{\begin{itemize}}
\newcommand{\ei}{\end{itemize}}
\newcommand{\bbbm}{\begin{pmatrix}}
\newcommand{\eeem}{\end{pmatrix}}

\newcommand{\bbx}{{\bf x}}

\newcommand{\cH}{{\cal H}}

\newcommand{\cP}{{\cal P}}

\newcommand{\N}{{\mathbb N}}

\newcommand{\R}{{\mathbb R}}

\newcommand{\de}{\delta}

\newcommand{\ep}{\epsilon}

\newcommand{\la}{\lambda}

\newcommand{\vp}{\varphi}

\newcommand{\ignore}[1]{}{}
\newcommand{\noin}{\noindent}

\newcommand{\nn}{\nonumber}

\newcommand{\p}{{\partial}}

\newcommand{\q}{\quad}

\newcommand{\dom}{\mathop{\rm dom}}
 \newcommand{\Id}{\mathop{\rm Id}}

\newcommand{{\QED}}{{\hfill QED} \bigskip}

\renewcommand{\subset}{\subseteq}

\newcommand{\cal}{\mathcal}

\renewcommand{\xy}{\langle x,  y \rangle}
\newcommand{\yz}{\langle y,  z \rangle}
\newcommand{\zx}{\langle z,  x \rangle}
\newcommand{\xx}{\langle x_i,  x_j \rangle}

\DeclareMathOperator{\id}{Id}

\definecolor{darkspringgreen}{rgb}{0.09, 0.45, 0.27} 
\definecolor{darkgray}{rgb}{0.66, 0.66, 0.66}

\numberwithin{equation}{section}
\numberwithin{theorem}{section}
\tikzset{
    partial ellipse/.style args={#1:#2:#3}{
        insert path={+ (#1:#3) arc (#1:#2:#3)}
    }
}

\newcommand{\prox}{\ensuremath{\operatorname{Prox}}}

\begin{document}
\title[Maximality and involutivity of multi-conjugate convex functions]
{
Maximal monotonicity and cyclic involutivity of multi-conjugate convex functions
}

\thanks{
\em The author wishes to express gratitude to the Korea Institute of Advanced Study (KIAS) AI research group and the director Hyeon, Changbong for their hospitality and support during his stay at KIAS in 2022, where parts of this work were performed. TL would also like to thank Robert McCann for the fruitful discussions.
}

\date{\today}

\author{Tongseok Lim}
\address{Tongseok Lim: Krannert School of Management \newline Purdue University, West Lafayette, Indiana 47907, USA}
\email{lim336@purdue.edu}

\begin{abstract} 
A cornerstone in convex analysis is the crucial relationship between functions and their convex conjugate via the Fenchel-Young inequality. In this dual variable setting, the maximal monotonicity of the contact set $ \big\{(x,y) \ \big| \ f(x) + f^*(y) = \langle x,y \rangle  \big\}$ is due to the involution $f^{**} = f$ holding for convex lower-semicontinuous functions defined on any Hilbert space. 

We investigate the validity of the cyclic version of involution and maximal monotonicity for multiple (more than two) convex functions. As a result, we show that when the underlying space is the real line, cyclical involutivity and maximal monotonicity induced by multi-conjugate convex functions continue to hold as for the dual variable case. On the other hand, when the underlying space is multidimensional, we show that the corresponding properties do not hold in general unless a further regularity assumption is imposed.  We provide detailed examples that illustrate the significant differences between dual- and multi-conjugate convex functions, as well as between uni- and multi-dimensional underlying spaces.
\end{abstract}

\maketitle


\section*{Definitions and Assumptions}
\noin{\boldmath$\cdot$} Unless specified otherwise, all functions have their values in $\R \cup \{+\infty\}$, i.e., functions do not assume the value $-\infty$ (only exception is in Lemma \ref{switch}). And all functions are assumed to be (or verified to be) proper, that is, $f \not\equiv +\infty$ but there exists $x \in \cH$ such that $f(x) \in \R$, unless it is stated as $f \equiv +\infty$.
\\
\noin{\boldmath$\cdot$} $\cH$ represents a real Hilbert space equipped with an inner product $\langle \ , \ \rangle$ on which all functions in the paper are defined. $|x| = \sqrt{\langle x , x \rangle}$ denotes the norm.
\\
\noin{\boldmath$\cdot$} $\langle x , y \rangle$ represents the dot product if $\cH= \R^n$, and is also denoted by $x \cdot y$.
\\
\noin{\boldmath$\cdot$} ${\cal A}(\cH)$ denotes the set of proper, lower-semicontinuous and convex functions on $\cH$.
\\
\noin{\boldmath$\cdot$} For $N \in \N$ and $\Gamma \subset \cH^N$, $S(\Gamma) := \big\{ \sum_{i=1}^N x_i \in \cH \ \big| \ (x_1,x_2,...,x_N) \in \Gamma \big\}$.
\\
\noin{\boldmath$\cdot$} For $N \in \N$, $\Delta = \Delta_{\cH^N} := \{(x,...,x) \in \cH^N \ | \ x \in \cH \}$ denotes the diagonal subspace of $\cH^N$.
\\
\noin{\boldmath$\cdot$} $\bbx = (x_1,...,x_N)$ denotes an arbitrary element in the product space $\cH^N$.
\\
\noin{\boldmath$\cdot$} $c=c_N : \cH^N \to \R$ denotes the ``cost function"; \ $c(\bbx)  := \sum_{1\leq i<j\leq N}\langle{x_i}, {x_j} \rangle$.
\\
\noin{\boldmath$\cdot$} For proper functions $\{f_i\}_{i=1,...,N}$ satisfying $
\sum_{i=1}^N f_i(x_i) \ge c(\bbx)$ for all $\bbx \in \cH^N$, we let $\Gamma = \Gamma_{\{f_i\}_{i=1}^N} := \big\{ \bbx\in \cH^N \ \big| \  \sum_{i=1}^N f_i(x_i) = c(\bbx) \big\}$.
\\
\noin{\boldmath$\cdot$} For proper functions $f_1,...,f_N$ on $\cH$, $\big(\bigoplus_{i=1}^N f_i\big)(\bbx) :=  \sum_{i=1}^N f_i(x_i)$.
\\
\noin{\boldmath$\cdot$}  $\square$ denotes the infimal convolution; \ $(f\square g)(x):=\inf_{y\in {\cal H}}\big(f(y)+g(x-y)\big)$.
\\
\noin{\boldmath$\cdot$} $q$ denotes the quadratic function; \ $q(x) := \frac12 |x|^2$.
\\
\noin{\boldmath$\cdot$} $e_f := f \square q$ denotes the Moreau envelope of $f$.
\\
\noin{\boldmath$\cdot$} For $f \in {\cal A}(\cH)$, $\prox_f$ denotes the proximal mapping (\cite[Definition 12.23]{BC2017}).
\\
\noin{\boldmath$\cdot$} $\Id : \cH \to \cH$ denotes the identity mapping; \ $\id(x)=x$.
\\
\noin{\boldmath$\cdot$} For $A \subset \cH$, int$A$ denotes the interior of $A$.
\\
\noin{\boldmath$\cdot$} For a function $f$, $\dom f := \{x \in \cH \ | \ f(x) \in \R\}$ denotes the domain of $f$.
\\
\noin{\boldmath$\cdot$} For $\la > 0$, a function $f$ is called $\la$-strongly convex if $f - \la q$ is convex.
\\
\noin{\boldmath$\cdot$} For $\Gamma \subset \cH^N$ and $i_0 \in \{1,...,N\}$, let $\Gamma_{i_0}: \cH \to 2^{\cH}$ denote the set-valued mapping; \  
$
\Gamma_{i_0} (x_{i_0}) :=\big\{ \sum_{i \neq i_0}x_i \ \big|\ (x_1,...,x_{i_0},...,x_N)\in \Gamma\ \big\}.
$
\\
\noin{\boldmath$\cdot$} (Essential smoothness \cite{Rockafellar1}) $f \in {\cal A}(\cH)$ is called essentially smooth if it satisfies the following three conditions for $C = {\rm int}(\dom f)$:

(a) $C$ is not empty;

(b) $f$ is differentiable thoughout $C$;

(c) $\lim_{i \to \infty} | \nabla f(x_i) | = +\infty$ whenever $x_1,x_2,...$ is a sequence in $C$ converging to a boundary point $x$ of $C$.
\\
\noin{\boldmath$\cdot$} ($c$-cyclical monotonicity \cite{BBPW}) 
Let $n,N \in \N \setminus \{1\}$.	The subset $\Gamma$ of $\cH^N$ is said to be $c$-cyclically monotone of order $n$, $n$-$c$-monotone for short, if for all $n$ tuples $(x^1_1,\dots,x_N^1),\dots,(x_1^n,\dots,x_N^n)$ in $\Gamma$ and every $N$ permutations $\sigma_1,\dots,\sigma_N$ in $S_n$, the following holds:
	\begin{equation}\label{cycmondef}
	\sum_{j=1}^nc(x_1^{\sigma_1(j)},\dots,x_N^{\sigma_N(j)})\leq \sum_{j=1}^n c(x_1^{j},\dots,x_N^{j}).
	\end{equation}
	$\Gamma$ is said to be $c$-cyclically monotone if it is $n$-$c$-monotone for every $n\in \N \setminus \{1\}$, and $\Gamma$ is said to be $c$-monotone if it is $2$-$c$-monotone. Finally, $\Gamma$ is said to be maximally $n$-$c$-monotone if it has no proper $n$-$c$-monotone extension.

\section{Introduction and our contribution}

Over the last few decades, the theory and applications of duality in convex analysis and monotone operator theory have advanced significantly (\cite{BC2017, Rock-Wets}), the study of dual variables $x, y \in \cH$, a function $f$, and its convex conjugate
\be
f^*(y) := \sup_{x \in \cH}\, \langle x, y \rangle - f(x)
\ee
which are linked by the Fenchel-Young inequality (see \cite[Proposition 13.15]{BC2017})
\be
f(x) + f^*(y) \ge \xy, \q x,y \in \cH.
\ee
We refer to  \cite{BNO2003, Bot2010, AAM2009, AAM2011, AASW2021, BS2008} for recent treatment of duality in optimization and convex analysis. Studies showed that if $f, g$ are convex conjugate to each other, i.e., if $f^* = g$ and $g^* = f$, then the following contact set 
\begin{align*}
\Gamma &= \{(x,y) \in \cH^2 \ | \ f(x) + g(y) = \xy  \} \\
&= \{(x,y) \in \cH^2 \ | \ y \in \p f (x) \} \\
&= \{(x,y) \in \cH^2 \ | \ x \in \p g(y) \}
\end{align*}
is {\em maximally monotone} (Rockafellar \cite{Rockafellar2}), which property turns out to be, by Minty's theorem (see \cite[Theorem~21.1]{BC2017}),
 equivalent to
 \be\label{maximality}
S(\Gamma) = \cH
\ee 
that is, $S(\Gamma)$ is not a proper subset of $\cH$. Note that if $\Gamma \subset \cH \times \cH$ is {\em monotone}, i.e.,
\be
\langle x-y, u-v \rangle \ge 0 \ \text{ for any } \ (x,u) \in \Gamma, (y,v) \in \Gamma, \nn
\ee
then $\Gamma \cap (\Delta^\perp + p)$ is either empty or a singleton for every $p \in \Delta$ (see \cite[Corollary 2.4]{BBPW}), where $\Delta$ is the diagonal subspace of $\cH \times \cH$. Thus \eqref{maximality} is a maximality assertion, saying there is no ``hole" in the monotone set $\Gamma$. 

These results established by Minty, Rockafellar, Fenchel, Moreau, and others have laid the groundwork for modern theory of nonlinear monotone operators \cite{C2018, RB1, RB2}.

Recently, S. Bartz, H.H. Bauschke, H.M. Phan, and X. Wang \cite{BBPW} provided a significant extension of the convex analysis theory into the multivariate, or multi-marginal (i.e., variables more than two), situation, stating that ``a comprehensive multi-marginal monotonicity and convex analysis theory is still missing." We denote $\bbx = (x_1,...,x_N) \in \cH^N$ and define $c_N : \cH^N \to \R$ by
\begin{equation}\label{cost}
	c(\bbx) = c_N(\bbx) := \sum_{1\leq i<j\leq N}\langle{x_i}, {x_j} \rangle.
\end{equation}
We may simply denote by $c$. An analogous statement of the bi-conjugacy --- $f^* = g$ and $g ^* = f$ --- in the multivariate setting can now be given as follows.

\begin{definition}[$c$-conjugate tuple \cite{BBPW}]
	For each $1\leq i\leq N$, let $f_i: \cH \to \R \cup \{+\infty\}$ be a proper function, i.e., not entirely $+\infty$. We say that $(f_1,\ldots,f_N)$ is a {\em $c$-conjugate tuple} if for each $1\leq i_0\leq N$ and $x_{i_0} \in \cH$, we have
	\begin{equation*}
	f_{i_0}(x_{i_0})=\Big(\bigoplus_{i\neq i_0}f_i\Big)^c(x_{i_0}) :=\sup_{i\neq i_0,\, x_i\in \cH}\ c(x_1,...,x_{i_0},...,x_N)-\sum_{i\neq i_0}f_i(x_i).	\end{equation*} 
\end{definition}

Now we focus on the following theorem, which is a culmination of the results in \cite{BBPW} (see Theorems 2.5 and 4.3 in \cite{BBPW}), significantly extending the theory of maximal monotone operators induced by several  convex functions.

\begin{theorem}[\cite{BBPW}]\label{t:split_max_c-mono}
For $1\leq i\leq N$, suppose $f_i \in {\cal A}(\cH)$ satisfy 
\[
\sum_{i=1}^N f_i(x_i) \ge c(\bbx) \ \text{ for all }\ \bbx \in \cH^N.
\]
Let $\Gamma = \Gamma_{\{f_i\}_{i=1}^N}$. The following assertions are equivalent:
	\begin{enumerate}
		\item \label{t:split_max_c-mono-i}
 There exist $1\leq i_0\leq N$ such that $\Gamma_{i_0}$ is maximally monotone; 
				
		\item\label{t:split_max_c-mono-ii} There exist $1\leq i_0\leq N$ such that $\Gamma_{i_0}=\partial f_{i_0}$;
		
		\item\label{t:split_max_c-mono-iii} $\Gamma_{i}=\partial f_{i}\ $ for each $1\leq i\leq N$;
		
		\item\label{t:split_max_c-mono-iv} $\prox_{f_1}+\cdots+\prox_{f_N}=\Id$;
		
		\item\label{t:split_max_c-mono-v} $e_{f_1^*}+\cdots+e_{f_N^*}=q$;		
		\item\label{t:split_max_c-mono-vi} $\Gamma+\Delta^\perp= \cH^N$;
		
                 \item\label{t:split_max_c-mono-vii} $S(\Gamma)=\cH$.

\end{enumerate}	
In this case, $(f_1,\ldots,f_N)$ is a $c$-conjugate tuple, and $\Gamma$ determines $(f_1,\ldots,f_N)$\\ uniquely up  to an additive constant tuple $(\rho_1,\ldots,\rho_N)$ such that $\sum_{i=1}^N \rho_i=0$.
\end{theorem}

Here \ref{t:split_max_c-mono-iv} represents the partition of the identity into a sum of firmly nonexpansive mappings, and \ref{t:split_max_c-mono-v} represents Moreau’s decomposition of the quadratic function into envelopes in the multivariate settings. In addition, \cite{BBPW} shows that $\Gamma$ is {\em maximally $c$-monotone} and, consequently, {\em maximally $c$-cyclically monotone} if any of the  assertions~\ref{t:split_max_c-mono-i}--\ref{t:split_max_c-mono-vii} hold. 

 It is natural to ask whether the cyclical conjugacy of $(f_1,...,f_N)$ conversely yields the assertions as well. As a result, \cite{BBPW} addresses the following question:
 \begin{question}\label{question1}
For any proper $c$-conjugate tuple  $\{f_i\}_{i=1}^N$, is $\Gamma = \Gamma_{\{f_i\}_{i=1}^N}$ maximal in the sense of \eqref{maximality} (so that the statements~\ref{t:split_max_c-mono-i}--\ref{t:split_max_c-mono-vii} hold)?
\end{question}
According to \cite{BBPW}, this question is still open to the best of their knowledge. The question appears to be quite interesting, particularly because it is well known to be affirmative when $N=2$ by the work of Rockafellar and Minty. 

In this regard, the following partial result is given in \cite[Theorem 4.6]{BBPW}.

\begin{theorem}[\cite{BBPW}]\label{3-marginal smooth conjugate}
	Let $N=3$, $n \in \N$ and $g, h \in {\cal A}( \R^n)$. Suppose that $f=(g\oplus h)^c$ and that $f$ is essentially smooth. Let $\Gamma = \Gamma_{\{f,g,h\}}$. Then assertions~\ref{t:split_max_c-mono-i}--\ref{t:split_max_c-mono-vii}  
	in Theorem~\ref{t:split_max_c-mono} hold (and therefore $\Gamma$ is maximally $c$-monotone).
\end{theorem}

While Theorem \ref{3-marginal smooth conjugate} gives the first and affirmative answer to Question \ref{question1}, it appears to be restricted in the following ways: The value of $N$ must be $3$, and one of the $c$-conjugate convex functions must be smooth. In this regard, our first result addresses the problem of relaxing $N$ to be arbitrary as follows.

\begin{theorem}\label{N-marginal smooth conjugate}
	Let $n,N \in \N$, $N \ge 2$ and $f_1,...,f_N \in {\cal A}( \R^n)$. Suppose that $f_1=(\bigoplus_{i=2}^N f_i)^c$ and that $f_1$ is essentially smooth. Let $\Gamma = \Gamma_{\{ f_1,...,f_N\}}$. Then assertions~\ref{t:split_max_c-mono-i}--\ref{t:split_max_c-mono-vii}  
	in Theorem~\ref{t:split_max_c-mono} hold true.\end{theorem}
However, this result is not yet entirely satisfactory because the smoothness assumption must be imposed on one of the $c$-conjugate functions, which is not necessary when $N=2$. Can the theorem be proved in full generality, i.e., for any $c$-conjugate convex functions that only belong in ${\cal A}(\R^n)$?

Our next result shows that it is indeed affirmative if $n=1$., i.e., if $\cH = \R$.
\begin{theorem}\label{mainx}
	Let $N \ge 2$ and assume that $f_1,...,f_N \in {\cal A}( \R)$ are $c$-conjugate. Let $\Gamma = \Gamma_{\{f_1,...,f_N\}}$. Then assertions~\ref{t:split_max_c-mono-i}--\ref{t:split_max_c-mono-vii}  
	in Theorem~\ref{t:split_max_c-mono} hold true.
	\end{theorem}
	
While working on these findings, the author realizes that Question \ref{question1} is in fact closely related to the following notion of {\em cyclical involutivity}, which, to the best of the author's knowledge, has not been introduced or investigated. As a result, we are led to provide the following definition.

\begin{definition}[Cyclical involutivity] Let $N \in \N$, $N \ge 2$, $c= c_N$. We shall say that a subset of functions $\Omega \subset {\cal A}(\cH)$ is $N$-cyclically involutive if for any $f_1,f_2,...,f_{N} \in \Omega
$ that satisfies 
\be\label{conjugate}
f_i(x_i) =\Big(\bigoplus_{j\neq i}f_j\Big)^c(x_{i}) \ \text{ for every } \  i=2,...,N,
\ee
then	 $(f_1,\ldots,f_N)$ is a $c$-conjugate tuple, that is, $f_1$ also satisfies
\be\label{wts}
f_1(x_1) =\Big(\bigoplus_{j\neq 1}f_j\Big)^c(x_{1}).
\ee
$\Omega$ is called cyclically involutive if it is $N$-cyclically involutive for every $N \ge 2$.
\end{definition}
Indeed, this definition is inspired by the following cornerstone statement by Fenchel and Moreau (see \cite[Theorem 13.37]{BC2017})
\be
f^{**} = f \ \text{ for every } f \in {\cal A}(\cH),
\ee
which can be rephrased as ``${\cal A}(\cH)$ is $2$-cyclically involutive" for any Hilbert space $\cH$. It is natural to wonder whether ${\cal A}(\cH)$ is $N$-cyclically involutive for each $N \ge 3$ as well. Our next result shows that it is also affirmative if $\cH = \R$.
\begin{theorem}\label{mainy}
${\cal A}(\R)$ is cyclically involutive.
\end{theorem}

See the proof of Theorem \ref{main}, in which we prove Theorems \ref{mainx} and \ref{mainy} simultaneously using an induction on $N$, demonstrating the close relationship between cyclical involutivity and maximal monotonicity.

Now for $\cH = \R^n$, we readily see that Theorems \ref{t:split_max_c-mono}, \ref{N-marginal smooth conjugate} combine to imply:
\begin{corollary}\label{main3}
The class of essentially smooth functions in ${\cal A}(\R^n)$ is cyclically involutive.
\end{corollary}

Is ${\cal A}(\R^n)$ itself cyclically involutive as well? Somewhat surprisingly, it turns out not to be true in general, as the following counterexample indicates.

\begin{proposition}\label{noinvolutive}
There exists $f,g,h \in {\cal A}(\R^2)$ such that $f = (g \oplus h)^c$, $g = (h \oplus f)^c$, but $h \neq (f \oplus g)^c$.
\end{proposition}

This proposition makes Question \ref{question1} obscure. Nevertheless, we can still hope that, if $f,g,h$ are assumed to be $c$-conjugate, $\Gamma_{\{f,g,h\}}$ might satisfy \eqref{maximality}. This sounds plausible in view of Theorem \ref{N-marginal smooth conjugate} where only one of the $c$-conjugate functions needs to be essentially smooth, which appears to be a pretty mild, and, seemingly redundant assumption when compared to the case $N=2$.

 This hope, however, is dashed by the following counterexample.

\begin{proposition}\label{counter1}
Let $\la > 0$, and $u =(1,0) \in \R^2$, $v = (\tfrac12, \tfrac{\sqrt{3}}2) \in \R^2$. There exists a $c$-conjugate triple $f,g,h \in {\cal A}(\R^2)$  for which $\Gamma = \Gamma_{\{f,g,h\}}$ satisfies
\be\label{SGamma1}
S(\Gamma) = \R^2 \setminus {\rm int}\, {\mathbf H}_\la \nn
\ee
where ${\mathbf H}_\la$ is the convex hull of its six vertices $\pm 2\la u$, $\pm 2\la v$, and $\pm 2\la (v - u)$. Furthermore, $\Gamma$ is not maximally $c$-monotone.
\end{proposition} 
This implies that the regularity assumption in Theorem \ref{N-marginal smooth conjugate} is not redundant.

It is worth noting that \cite{BBPW} leaves the open question of whether every contact set $\Gamma$ generated by the $c$-conjugate tuple $(f_1,\ldots,f_N)$ is maximally $c$-monotone. The proposition shows that the answer is negative unless $\cH = \R$, in which case Theorem \ref{mainx} provides an affirmative answer. More detailed statement of the proposition and related examples can be found in Section \ref{sec4}.

In summary, we provide a fairly complete answer to Question \ref{question1} by demonstrating the significant difference between the cases $N=2$ and $N \ge 3$, and also, between $\cH=\R$ and $\cH=\R^n$. While Question \ref{question1} holds true for $N=2$ and ${\cal A}(\cH)$ is $2$-cyclically involutive as shown by Fenchel, Moreau, Minty and Rockafellar, we show both statements fail  in general when $N \ge 3$. Nevertheless, Theorems \ref{mainx}, \ref{mainy} show that the results are still affirmative in full generality for any $N \ge 3$ if $\cH = \R$. Furthermore, Theorems \ref{N-marginal smooth conjugate}, \ref{main3} show that the results are also affirmative within the class of essentially smooth convex functions when $\cH = \R^n$. Finally, we believe that identifying other significant subsets of ${\cal A}(\cH)$ that are cyclically involutive and understanding their relationship with maximal monotonicity is an interesting question for future research.
\\

\subsection{Connection with theory of multi-marginal optimal transport} Let $(X_1,\mu_1),\ldots,(X_N,\mu_N)$ be
Borel probability spaces, and $X:= X_1\times \cdots \times X_N$. Denote by $\Pi(X)$ the set of all Borel probability measures $\pi$ on $X$ whose marginals are the $\mu_i$'s \cite{San, Vil}. Given a cost function $c : X \to \R$, the optimal transport problem refers to the following optimization problem:
\be\label{OT}
P_c := \inf_{\pi\in\Pi(X)} \int_X c(x)d\pi(x).
\ee
To distinguish it from the two-marginal case, the problem is commonly referred to as multi-marginal optimal transport problem when $N \ge 3$. In this problem, the optimal transport cost $P_c$, as well as the geometry and structure of optimal transport plans -- the solutions to \eqref{OT} -- are sought. It is well known that the problem \eqref{OT} is attained, i.e., the OT plans exist, under suitable assumptions on the cost and marginals.

Because \eqref{OT} is an infinite-dimensional linear programming problem, it has a dual problem whose formulation turns out to have the following form:
\be\label{dualOT}
D_c := \sup_{\begin{array}{c}
	f_i\in L_1({\mu_i}),\\ 
	\sum_{1\leq i\leq N} f_i(x_i) \leq c(\bbx)	
	\end{array}} \sum_{1\leq i\leq N}\int_{X_i}f_i(x_i)d\mu_i(x_i).
\ee
Kellerer's~\cite{Kel} generalization
of the Kantorovich duality states that, under mild assumptions on the marginals $\mu_1,...,\mu_N$ and cost function $c$, it holds
\be
P_c = D_c.
\ee
This has the important implication that every optimal transport $\pi$ solving \eqref{OT} is concentrated on the contact set
\be\label{contactset}
\Gamma = \bigg\{ \bbx \in X \ \bigg| \ \sum_{1\leq i\leq N} f_i(x_i) = c(\bbx) \bigg\} 
\ee
where $(f_1,\dots, f_N)$ is a solution to the dual problem \eqref{dualOT}, which is also attained under suitable assumptions. This provides a crucial information for investigating the geometry of optimal transport plans.

The interaction between the optimal transport and its dual problems is what makes the theory surprisingly powerful for many applications in fields such as analysis, geometry, PDEs, probability, statistics, economics, data sciences, and many researchers have helped to advance the field \cite{GanMc, GhMo, Mc, RR1, RR2, San, Vil}. Regarding the geometry of optimal transport for $N=2$, arguably one of the most well-known and widely applied result is the Brenier's theorem \cite{Bre}: given marginals $\mu_1, \mu_2 \in \cP_2(\R^n)$ and the cost function $c(x_1,x_2) = |x_1 - x_2|^2$, there exists a convex function $\vp$ such that for any solution $\pi$ to \eqref{OT}, it holds
\be\label{Brenier}
y \in \p \vp (x) \ \ \pi-a.e.\, (x,y), \text{ and moreover, } y= \nabla \vp(x) \text{ if $\mu_1$ has density.} 
\ee
Because the geometry of the subdifferential $\p \vp$ is well understood by studies in convex analysis, Brenier's theorem could yield important further results. Likewise, a better understanding of multi-conjugate convex analysis should also have a significant impact on the theory of multi-marginal optimal transport, their geometry, and applications. This is one of the motivations of this paper.

Recent advances in the theory of multi-marginal optimal transport and its geometrical structures have been rapid and fruitful, yielding numerous new research directions and open problems \cite{AC, BBW2, CGC, GS, KP, KP2, KP3, MPC, Nen, Pas2, Pas3}. In light of  two-marginal optimal transport theory, Brenier's theorem and their consequences, it appears clear that understanding the geometry of the contact set \eqref{contactset} is crucial, much of which falls within the scope of the multi-conjugate convex analysis. In this regard, the author hopes that this paper will contribute to a better understanding of the geometry of various multi-marginal optimal transport problems.

\subsection{Organization of the paper} 
This paper is organized as follows. In Section \ref{sec2}, we prove Theorem \ref{N-marginal smooth conjugate}. In Section \ref{sec3}, we prove Theorems \ref{mainx} and \ref{mainy}. 
In Section \ref{sec4}, we explain Proposition \ref{counter1} and provide additional examples to demonstrate the difference between $\cH = \R$ and $\cH = \R^n$. Finally, in Section \ref{sec5}, we revisit cyclical involutivity and maximal monotonicity in the $N=3$ case in the context of duality, as well as provide details for Proposition \ref{noinvolutive}.

\section{Extension of Theorem \ref{3-marginal smooth conjugate}: Proof of Theorem \ref{N-marginal smooth conjugate}}\label{sec2}

We will need the following lemma whose proof is given in \cite[Proposition 14.19]{BC2017} for the case $N=1$. Extension for arbitrary $N \in \N$ appears to be useful.
\begin{lemma}\label{switch}
For $g, f_1,f_2,...,f_N \in {\cal A}(\cH)$, we have the following identity
\be
\sup_{\substack{x_i \in \dom f_i \\ i=1,...,N}} \bigg( g(x+x_1+...+x_N) - \sum_{i=1}^N f_i(x_i) \bigg) = \bigg(g^* - \sum_{i=1}^N f_i^* \bigg)^* (x) \ \  \text{for all } x \in \cH,
\ee
where $g^* - \sum_{i=1}^N f_i^* : \cH \to [-\infty, +\infty]$ is defined as (note $-\infty$ can be assumed)
\be
\bigg(g^* - \sum_{i=1}^N f_i^*\bigg) (y) := \begin{cases}
\big(g^* - \sum_{i=1}^N f_i^* \big) (y) \ \text{ if }\, y \in \dom g^*,\\
+\infty  \ \text{ if }\, y \notin \dom g^*.
\end{cases}
\nn
\ee
\end{lemma}

\begin{proof} Recall $S(\bbx) := x_1+ \dots +x_N$. We can proceed as follows:
\begin{align*}
 \bigg(g^* - \sum_{i=1}^N f_i^* \bigg)^* (x) 
 &= \sup_{y \in \dom g^*} \langle x, y \rangle - \bigg(g^* - \sum_{i=1}^N f_i^*\bigg) (y) \\
 &= \sup_{y \in \dom g^*} \langle x, y \rangle - g^*(y) +  \sum_{i=1}^N \sup_{x_i \in \dom f_i}  \langle y, x_i \rangle  - f_i(x_i) \\
 &= \sup_{y \in \dom g^*} \langle x, y \rangle - g^*(y) +  \sup_{\substack{x_i \in \dom f_i \\ i=1,...,N}} \langle y, S(\bbx) \rangle - \sum_{i=1}^N f_i(x_i) \\
 &=  \sup_{\substack{x_i \in \dom f_i \\ i=1,...,N}} \sup_{y \in \dom g^*} \langle y , x + S(\bbx) \rangle - g^*(y) -  \sum_{i=1}^N f_i(x_i) \\
 &= \sup_{\substack{x_i \in \dom f_i \\ i=1,...,N}} g(x + S(\bbx)) -  \sum_{i=1}^N f_i(x_i).
 \end{align*}
 This proves the lemma.
\end{proof}

Our proof of Theorem \ref{N-marginal smooth conjugate} will closely follow the proof presented in \cite{BBPW}, utilizing Lemma \ref{switch} along with the following fact given in \cite{Sol}.

\noin\cite[Corollary 2.3]{Sol}. {\em Let $f:\R^n \to \R \cup \{\infty\}$ be proper and lower-semicontinuous. If $f^*$ is essentially smooth, then $f$ is convex.}

\begin{proof}[Proof of Theorem \ref{N-marginal smooth conjugate}.]
 Recalling $f_1=(\oplus_{i=2}^N f_i)^c$, we proceed
\begin{align}
&(f_1+q)(x_1) = \sup_{x_2,...,x_N \in\R^n} c(\bbx)+q(x_1) - \sum_{i=2}^N f_i(x_i) \nonumber\\
&=\sup_{x_2,...,x_N} x_2 \cdot (x_1+x_3+...+x_N) - f_2(x_2) + \sum_{\substack{1 \le j < k \le N \\ j \neq 2, k \neq 2}} x_j \cdot x_k + q (x_1) - \sum_{i=3}^N f_i(x_i)
\nonumber\\
&=\sup_{x_3,...,x_N} f_2^*(x_1+x_3+...+x_N) + q(x_1+x_3+...+x_N) - \sum_{i=3}^N (f_i+q)(x_i)
\nonumber\\
&= \bigg((f_2^*+q)^* - \sum_{i=3}^N (f_i +q)^*\bigg)^*(x_1)
\nonumber\\
&= \bigg( e_{f_2} - \sum_{i=3}^N e_{f^*_i} \bigg)^* (x_1)
 = \bigg( q - \sum_{i=2}^N e_{f^*_i} \bigg)^* (x_1) \nn
	\end{align}
where the last two equalities are due to Lemma \ref{switch} and Moreau's decomposition \eqref{t:split_max_c-mono-v} for $N=2$. Now since $f_1+q$ is essentially smooth, \cite[Corollary 2.3]{Sol} implies $q - \sum_{i=2}^N e_{f^*_k}$ is convex. And $e_{f^*_k}$, which is a Moreau envelope of a function in ${\cal A}(\R^n),$ is continuous in $\R^n$. Consequently,
\be
e_{f_1^*}=(f_1+q)^*=(q - \sum_{i=2}^N e_{f^*_i})^{**}=q - \sum_{i=2}^N e_{f^*_i}, \nn
\ee
	that is, $\sum_{i=1}^N e_{f^*_i}=q$, which is the Moreau decomposition \eqref{t:split_max_c-mono-v}.
\end{proof}

\section{One-dimensional domain: Proof of Theorems \ref{mainx} and \ref{mainy} }\label{sec3}

We begin with a lemma about strong convexity and how it is inherited.
\begin{lemma}\label{laconvex}
Let $f,g,h$ be proper functions on a Hilbert space $\cH$, satisfying
\be
f(x) = \sup_{y \in \dom g} h(x+y)-g(y). \nn
\ee
If $h$ is lower-semicontinuous and $\la$-strongly convex,  then so is $f$.
\end{lemma}

\begin{proof}
Let $k = h - \la q$, which belongs to ${\cal A}(\cH)$, thus $k = k^{**}$. We compute
\begin{align*}
f(x) &= \sup_{y \in \dom g} h(x+y)-g(y) \\
& = \sup_{y \in \dom g} \sup_{z \in \dom k^*} \la q(x+y) + \langle x+y , z \rangle - k^*(z) - g(y) \\
&=\la q (x) +  \sup_{z \in \dom k^*} \big\{ \langle x,z \rangle - k^*(z) +  \sup_{y \in \dom g} \{ \langle y,\la x+z \rangle - (g- \la q) (y) \} \big\} \\
&= \la q (x) +  \sup_{z \in \dom k^*} \langle x,z \rangle - k^*(z) + (g - \la q)^* (\la x + z) \\
&=: \la q (x) + \xi (x).
\end{align*}
Observe $\xi$ is convex lower-semicontinuous as a supremum of such functions.
\end{proof}

Let $\Delta$ be the diagonal subspace of $\cH^N$. The concept of a  ``directional convex envelope" is clearly relevant to  the study of $c$-conjugate convex functions. We introduce the following definition as we are not aware of it appearing elsewhere.

\begin{definition}[$\Delta$-convex envelope]\label{deltadef}
Let $\bbx = (x_1,...,x_N) \in \cH^N$, $S(\bbx) = \sum_{i=1}^N x_i$. Let $f : \cH^N \to \R \cup \{+\infty\}$ be proper, satisfying $f( \bbx ) \ge \langle S(\bbx), y \rangle + b$ for  some $y \in \cH$, $b \in \R$. Then $g$ is called the $\Delta$-{\em \,convex envelope of} $f$ if $g$ is the largest convex lower-semicontinuous function on $\cH$ satisfying $f( \bbx) \ge g(S(\bbx))$.
\end{definition}

\begin{lemma}\label{deltalem}
Let $f$ satisfy the condition in Definition \ref{deltadef}. Then $g$ is the $\Delta$-convex envelope of $f$ if and only if $g^*(y) = \sup_\bbx \langle S(\bbx), y \rangle - f(\bbx)$.
\end{lemma}

\begin{proof} 
Let $h$ be any element in ${\cal A}(\cH)$. The following equivalence
\begin{align*}
f(\bbx) &\ge h(S(\bbx)) \ \text{ for every } \bbx \in \cH^N \\
\iff f(\bbx) & \ge \langle S(\bbx), y \rangle - h^*(y) \ \text{ for every } \bbx \in \cH^N, y \in \cH \\
\iff h^*(y) &\ge \sup_\bbx \langle S(\bbx), y \rangle - f(\bbx)  \ \text{ for every } y \in \cH
\end{align*}
yields the lemma, since maximality of $g$ corresponds to minimality of $g^*$.
\end{proof}

\begin{definition} For proper functions $f,g,h : \cH \to \R \cup \{+\infty\}$, we say that $f$ and $g$ are $h$-conjugate if the following holds:
\be\label{conj}
f(x) = \sup_{y \in \dom g} h(x+y) - g(y), \q g(y) = \sup_{x \in \dom f}  h(x+y) - f(x).
\ee
\end{definition}

\begin{proposition}\label{envelope}
Assume that $f,g,h \in {\cal A}(\R)$ and that $f,g$ are $h$-conjugate. Assume further that $f$ and $h$ are continuous, and that $h$ is $\la$-strongly convex for some $\la >0$. Then $h$ is the $\Delta$-convex envelope of $f \oplus g$.
\end{proposition}

\begin{proof} By Lemma \ref{laconvex}, $f$ and $g$ are $\la$- stongly convex. In particular, $\bigcup_{x \in \R} \p f (x) = \bigcup_{y \in \R} \p g (y) = \R$. 
Firstly, we claim that the following set $\cal I$ is dense in $\R$:
\begin{align}\label{claim7}
\cal I &= \{ s \in \R \ | \ s = x+y \text{ such that } \p f(x) \cap \p g(y) \ne \emptyset, \text{ and} \\
&\text{either $f$ is differentiable at $x$ or $g$ is differentiable at $y \}$}. \nn
\end{align}
Let us prove the claim later. Let $H$ denote the $\Delta$-convex envelope of $f \oplus g$. Then $H \ge h$ since $f(x) + g(y) \ge h(x+y)$ by \eqref{conj}. The proposition asserts $H=h$. To prove this, we claim that it is sufficient to prove the following tightness assertion: For any $x_0 ,y_0 \in \R$ such that $\p f(x_0) \cap \p g(y_0) \ne \emptyset$ and either $f$ is differentiable at $x_0$ or $g$ is differentiable at $y_0$, we have
\begin{align}\label{claim3}
f(x_0) + g(y_0) = h(x_0 + y_0).
\end{align}
The sufficiency is because \eqref{claim3} implies $H=h$ on $\cal I$, and thus for any $s \in \R$, by the first claim, there exists a sequence $s_n$ in $\cal I$ such that $\lim s_n = s$, and
\be
H(s) \le \liminf H(s_n) = \liminf h(s_n) = h(s) \nn
\ee
as desired, thanks to the continuity of $h$.

Now to verify \eqref{claim3}, by translation, we may assume without loss of generality that $x_0 = y_0= 0$. Moreover we may assume that $0 \in  \p f(0) \cap \p g(0)$. To see why this can be assumed, let $a \in  \p f(0) \cap \p g(0)$. Consider $\tilde f(x) = f(x) - \langle a, x \rangle -f(0)$, $\tilde g(y) = g(y) - \langle a, y \rangle -g(0)$, and $\tilde h(z) = h(z) - \langle a, z \rangle -f(0) - g(0) $. Then $f$ and $g$ are $h$-conjugate if and only if $\tilde f$ and $\tilde g$ are $\tilde h$-conjugate. And since $\min \tilde f =  \tilde f (0) = 0$ and $\min \tilde g =  \tilde g (0) = 0$,  \eqref{claim3} holds if and only if $\tilde h(0)=0$. Our discussion so far indicates it is sufficient to prove the claim \eqref{claim3} under the assumption $x_0=y_0=0$, $f, g$ are $h$-conjugate, $f(0) = \min f = g(0) = \min g = 0$, and either $f$ or $g$ is differentiable at $0$; and the goal is to prove $h(0)=0$. 

Now to derive a contradiction, suppose $h(0) < 0$, so that $m = \min h < 0$. Assume $f$ is differentiable at $0$  (the proof will be the same in the case $g$ is differentiable at $0$, by switching the role of $f$ and $g$). Let $K = \{ x \ | \ h(x) \le m/2 \}$. Since $h$ is strongly convex, there exists $\de > 0 $ such that 
\be\label{grad}
x \in \R \setminus K \text{ and }  z \in \p h(x) \ \text{ implies } \ |z| \ge \de.
\ee
By \eqref{conj}, given $\ep > 0$, there exists $y_\ep$ such that $-\ep < h(y_\ep) - g(y_\ep) \le 0 = f(0)$. Now $f(x) \ge h(x+ y_\ep) - g(y_\ep)$ for all $x$ and $\nabla f (0) = 0$ imply that for all sufficiently small $\ep$, we must have $y_\ep \in K$ by \eqref{grad}. However, we then have
\be
- \ep < h(y_\ep) - g(y_\ep) \le h(y_\ep) \le m/2, \nn
\ee
a contradiction for small $\ep$, proving the tightness assertion.

It remains to prove $\cal I$ is dense in $\R$. For any $s \in \R$, recall that there exists $x,y \in \R$ such that $s \in \p f (x) \cap \p g(y)$. By translation and subtracting affine functions as before, we may assume $0 \in \p f (0) \cap \p g (0)$. Then notice the desired denseness will follow if we can show that for any $r > 0$, there exists $x, y \in \R$ such that $|x| <r$, $|y|<r$, $\p f(x) \cap \p g(y) \ne \emptyset$, and either $f$ is differentiable at $x$ or $g$ is differentiable at $y$. 

Now to prove the claim, assume that neither $f$ nor $g$ is differentiable at $0$, since otherwise there is nothing to prove. $\p f(0)$ is a compact interval, say $[a,b]$, since $f$ is continuous. And $W:= \p g ( (-r,r))$ is an open interval since $g$ is strongly convex, thus $g^*$ is differentiable and $\p g^* = (\p g)^{-1}$ is continuous.

Suppose $b \in W$. Then for any $\ep >0$, there exists $x \in [0, \ep)$ such that $f$ is differentiable at $x$ and $\nabla f (x) \in [b, b+\ep)$ since $f$ is differentiable a.e.. Thus $\nabla f(x) \in W$ for small $\ep$, and this implies $\nabla f (x) \in \p g(y)$ for some $y \in (-r,r)$, proving the claim. Likewise, the claim holds in the case $a \in W$. Finally, if $\{a,b\} \cap W = \emptyset$, then $W \subset (a,b)$, yielding that $g$ is Lipschitz and differentiable a.e. in $(-r,r)$. Hence there exists $y \in (-r,r)$ such that $\nabla g(y) \in \p f (0)$. This proves the desired denseness of $\cal I$ in $\R$, hence the proposition.
\end{proof}

\begin{remark}\label{remark1}
The above proof shows that Proposition \ref{envelope} will continue to hold for functions $f,g,h$ defined on an arbitrary Hilbert space $\cH$ if the corresponding set $\cal I$ is dense in $\cH$ given the mutual conjugacy \eqref{conj}.
\end{remark}

Now we combine Theorems \ref{mainx} and \ref{mainy} into the following theorem. 
\begin{theorem}\label{main}
Let $\cH = \R$, and let $c: \R^N \to \R$ be given by \eqref{cost}. Suppose $f_i \in {\cal A}(\R)$, $i=1,...,N$, satisfy
\be\label{conjugate}
f_i(x_i) =\Big(\bigoplus_{j\neq i}f_j\Big)^c(x_{i}) \ \text{ for every } \  i=2,...,N.
\ee
Then	 $(f_1,\ldots,f_N)$ is a $c$-conjugate tuple, that is, $f_1$ also satisfies
\be\label{wts}
f_1(x_1) =\Big(\bigoplus_{j\neq 1}f_j\Big)^c(x_{1}).
\ee	
In this case, the contact set $\Gamma = \Gamma_{ \{f_i\}_{i=1}^N}$ satisfies \eqref{maximality}, and thus, the assertions \ref{t:split_max_c-mono-i}--\ref{t:split_max_c-mono-vii} in Theorem \ref{t:split_max_c-mono} hold true if $\cH= \R$.
\end{theorem}

\begin{proof} If $N=2$, \eqref{wts} is equivalent to the fact $f^{**} = f$ by Fenchel and Moreau, and \eqref{maximality} is also shown by Rockafellar and Minty. We will proceed by an induction on $N \ge 3$, thus when $\cH = \R$, firstly we will extend these results \eqref{wts}, \eqref{maximality} for $N=3$. However, since the proof presented below will be valid for an arbitrary Hilbert space $\cH$, in the sequel, we will denote $\cH$ (rather than $\R$) by the underlying space, though we will eventually assume $\cH = \R$.
\\

\noin{\bf Step 1: $N=3$ case.} To begin, let us write $(f,g,h) = (f_1,f_2,f_3)$. We have \[ f(x) + g(y) + h(z) \ge \xy + \yz + \zx,
\]
 hence
\be\label{relation1}
 \vp(x,y) := f(x) + g(y) - \langle x,y \rangle \ge h^*(x+y). 
\ee
We will show $h^*$ is the $\Delta$-convex envelope of $\vp$. Notice \eqref{relation1} is equivalent to
\be\label{relation2}
F(x) + G(y) \ge H(x + y)
\ee
where $F= f + q$, $G = g + q$, $H = h^* + q$, and we recall $q(s) = \frac12 |s|^2$. Observe that the conjugacy assumption \eqref{conjugate}, which reads 
\begin{align}
f(x) = \sup_y \{h^*(x+y) + \xy - g(y) \}, \ g(y) = \sup_x \{h^*(x+y) + \xy - f(x) \} \nn
\end{align}
implies that $F$ and $G$ are also $H$-conjugate (see \eqref{conj}).
We claim that $H$ is the $\Delta$-convex envelope of $F \oplus G$. Observe that this implies \eqref{wts} for $N=3$, because if $H$ is the $\Delta$-convex envelope of $F \oplus G$, then $h^*$ must be the $\Delta$-convex envelope of $\vp$ due to the equivalence between \eqref{relation1} and  \eqref{relation2}. But this precisely means that $h = (f \oplus g)^c$ by Lemma \ref{deltalem}, as desired.

Then as argued in the proof of Proposition \ref{envelope}, it is sufficient to prove the claim under the assumption that $F$ and $G$ are $H$-conjugate, $F(0) = 0 = \min F$, $G(0) = 0 = \min G$, and the goal is to show $H(0)=0 = \min H$. 

To this end, for each $R > 0$, define
\begin{align}
G_R(y) &= G(y) \ \text{ if } \ |y| \le R, \q G_R(y) = +\infty \ \text{ if } \ |y| > R, \nn \\
h^*_R(x) &= \sup_{|y| \le R} \xy - h(y), \q H_R = h^*_R + q, \nn \\
F_R(x) &= \sup_{y} \{ H_R(x+y) - G_R(y)\}, \nn \\
F_R^H(y) &= \sup_x \{ H_R(x+y) - F_R(x)\}, \nn \\
K_R(y) &= \sup_x \{ H_R(x+y) - F(x)\}. \nn
\end{align}
Assume $R$ is large enough so that $\min G_R = \min G$ and $h_R^*$ is proper. We have
\begin{align}
&h_R^* \text{ has Lipschitz constant at most $R$,} \label{fact2} \\
&F_R \text{ is locally Lipschitz and monotone increasing to } F \text{ as } R \to \infty, \label{fact3}\\
&F_R^H  \text{ converges pointwise to } G. \label{fact5}
\end{align}
\eqref{fact3} is because $H_R$ and $-G_R$ monotonically increase in $R$, and $F_R$ is locally Lipschitz because the supremum defining $F_R$ is taken over $|y| \le R$ only. And \eqref{fact5} is because $G_R \ge F_R^H \ge K_R$, $G_R$ decreases to $G$, and $K_R$ increases to $G$ in $R$. 
Now since $F_R$ and $F_R^H$ are $H_R$-conjugate, by Proposition \ref{envelope}, we have
\be\label{fact6}
\text{$H_R(x+y)$ is the $\Delta$-convex envelope of $F_R(x) + F_R^H(y)$.}
\ee
Recall that our goal is to show $ \min H = \min F + 
\min G$. It is clear that $ \min H \le \min F + 
\min G$. To show the reverse, since $H_R \le H$  and $\min H_R = \min F_R + \min F_R^H$ by \eqref{fact6}, and also $G_R \ge F_R^H \ge K_R$, it is enough to show
\be\label{claim2}
\min F_R \nearrow \min F \ \text{ and } \  \min K_R \nearrow \min G   \ \text{ as } \, R \to \infty.
\ee
The fact that $\min F = F(0) = 0$ and that $F$ is of the form $F = f + q$ implies $F \ge q$, thus $F^* \le q$. Since $F_R$ is increasing to $F$, $F_R^*$ is decreasing in $R$ and is bounded below by $F^*$. Let $F_\infty$ denote the limit of $F_R^*$. We claim  $F_\infty = F^*$. To see this, first notice $F_\infty$ is convex as a limit of convex functions. Secondly we have $F_\infty$ is continuous, because $F^* \le F_\infty \le F_R^*$ and $F_R$ is $1$-strongly convex by Lemma \ref{laconvex}, thus $q - F_R^*$ is  continuous and convex. This implies $F_\infty$ is bounded on every bounded subset of $\cH$, thus is locally Lipschitz, hence continuous on $\cH$. This yields $F_\infty = (F_\infty)^{**} = F^*$ as claimed. This implies in particular,
\[
\min F = - F^*(0) = - F_\infty(0) = - \lim_{R \to \infty} F_R^*(0) =  \lim_{R \to \infty} \min F_R.
\]
Similarly, $\min K_R \nearrow \min G$. This proves $\min H = H(0)=0$, and thus \eqref{wts}.

Now to prove \eqref{maximality}, fix any $s \in \cH$. By \eqref{maximality} holding for $N=2$, there exists $z \in \cH$ such that $h(z) + h^*(s-z) = \langle z, s - z \rangle$. This yields $ z \in \p h^*(s-z)$, implying $s \in \p H(s-z)$. Since $F, G, H$ are strongly convex, there exist unique $x, y, u \in \cH$ such that $s \in \p F(x) \cap \p G(y) \cap \p H( u)$. The fact that $H$ is the $\Delta$-convex envelope of $F \oplus G$ now yields $F(x) + G(y) = H(x+y)$ and $ s \in \p H(x+y)$. Then the uniqueness of $u$ implies $s-z = x+y$, that is, $ s = x+y+z$. Finally, 
\begin{align*}
&F(x) + G(y) = H(x+y) \\
\iff &f(x) + g(y) = h^*(x+y) + \xy \\
\iff &f(x) + g(y) = h^*(s-z) + \xy \\
\iff &f(x) + g(y) = \langle z, s - z \rangle -h(z) + \xy \\
\iff &f(x) + g(y) + h(z) = \xy + \yz + \zx.
\end{align*}
This proves the maximality \eqref{maximality}, hence the theorem for $N=3$. 
\\

\noin{\bf Step 2: induction on $N$.} From now on, we extend the proof for $N \ge 4$. We proceed by an induction on $N$. Suppose the theorem holds for $N-1$. Define
\begin{align*}
\tilde \bbx &= (x_3,...,x_N), \q S( \tilde \bbx) = \sum_{i=3}^{N} x_i, \q c(\tilde \bbx ) = \sum_{3 \le i < j \le N} \xx, \\
\psi(\tilde \bbx) &= \sum_{i=3}^{N} f_i(x_i) - c(\tilde \bbx ), \q \vp(x_1, x_2) = f_1(x_1) + f_2(x_2) - \langle x_1, x_2 \rangle,\\
g &= (f_1 \oplus f_2)^c,  \text{ i.e., }  g(y) = \sup_{x_1, x_2} \big\{ \langle y, x_1 + x_2 \rangle - \vp (x_1, x_2) \big \}.
\end{align*} 

From the inequality
$\sum_{i=1}^N f_i(x_i) \ge c(\bbx)$, we have 
\begin{align}\label{ineq0}
\psi(\tilde \bbx) &\ge g\big(S( \tilde \bbx)\big) = \sup_{x_1, x_2} \big\{ \langle S( \tilde \bbx), x_1 + x_2 \rangle - \vp (x_1, x_2) \big \}.
\end{align}
Now comes the crux of the observation:  \eqref{ineq0}, the induction hypothesis (on the cyclical involutivity), and the conjugacy \eqref{conjugate} (i.e., each of the $f_3,...,f_N$ is the smallest convex function satisfying \eqref{ineq0} given others) combine to imply
\be
(f_3,...,f_N, g^*) \text{ are } c(\tilde \bbx, y)-\text{conjugate, where } c(\tilde \bbx, y)= c(\tilde \bbx) + \langle S( \tilde \bbx), y \rangle. \nn
\ee
In other words, $g\big(S( \tilde \bbx)\big)$ is the $\Delta$-convex envelope of $\psi(\tilde \bbx)$, by Lemma \ref{deltalem}. 

This in turn implies $f_2 = (f_1 \oplus g)^c$ due to the following calculation:
\begin{align*}
f_2(x_2) &=\Big(\bigoplus_{j\neq 2}f_j\Big)^c(x_2) \\
&= \sup_{x_1, \tilde \bbx} \big\{  \langle x_1,x_2 \rangle  + \langle x_1+x_2,  S(\tilde \bbx) \rangle - f_1(x_1) - \psi(\tilde \bbx) \big\} \\
&= \sup_{x_1} \big\{  \langle x_1,x_2 \rangle - f_1(x_1) + \sup_{\tilde \bbx}  \{ \langle x_1+x_2,  S(\tilde \bbx) \rangle - \psi(\tilde \bbx)\} \big\} \\
&= \sup_{x_1, y} \big\{  \langle x_1,x_2 \rangle   - f_1(x_1) + \langle x_1+x_2,  y \rangle -g(y) \big\} \\
&= \sup_{x_1, y} \big\{  \langle x_2,x_1 + y \rangle  -  (f_1(x_1) +g(y) - \langle x_1, y \rangle ) \big\} \\
& = (f_1 \oplus g)^c (x_2)
\end{align*}
where the fourth equality is because $g$ is the $\Delta$-convex envelope of $\psi$. The induction hypothesis (or the theorem we established for $N=3$) now implies 
\begin{align*}
f_1 = (f_2 \oplus g)^c = \Big(\bigoplus_{j\neq 1}f_j\Big)^c(x_1)
\end{align*}
where the second equality is by an analogous calculation given above. This completes the proof of the cyclical involutivity \eqref{wts}.

Finally, we establish the maximality \eqref{maximality}. Recall \eqref{ineq0}, that is
\be
\psi( \tilde \bbx) \ge g ( S( \tilde \bbx) ) \ge \langle y, S( \tilde \bbx) \rangle - g^*(y) \ \text{ for every }  x_3,...,x_N,y. \nn
\ee
Fix any $s \in \cH$.  By the conjugacy of $(f_3,...,f_N, g^*)$ and the induction hypothesis on the maximality, there exists $\tilde \bbx^s = ( \tilde x^s_3,..., \tilde x^s_N)$ such that
\be\label{eq0}
\psi( \tilde \bbx^s) = g ( S( \tilde \bbx^s) ) = \langle s - S( \tilde \bbx^s), S( \tilde \bbx^s) \rangle - g^*(s - S( \tilde \bbx^s)).
\ee
Similarly, the conjugacy of $(f_1,f_2, g)$ yields 
\be
\vp(x_1,x_2) \ge g^*(x_1+x_2) \ge \langle z, x_1+x_2 \rangle - g(z) \ \text{ for every }  x_1,x_2, z, \nn
\ee
and for the same $s$, there exists $x^s_1, x^s_2 \in \cH$ such that
\be\label{eq2}
\vp(x^s_1,x^s_2) = g^*(x^s_1+x^s_2) = \langle  s - x^s_1 - x^s_2, x^s_1+x^s_2 \rangle - g(s - x^s_1 - x^s_2).
\ee
However, the pair $u,v \in \cH$ that satisfies $u+v = s$ and $g(u) + g^*(v) = \langle u,v \rangle$ is unique. Hence \eqref{eq0}, \eqref{eq2} implies $S( \tilde \bbx^s) = s - x^s_1 - x^s_2$, or $s = \sum_{i=1}^N x^s_i$. From this identity, by adding the identities \eqref{eq0}, \eqref{eq2}, we obtain
\begin{align*}
&\vp(x^s_1,x^s_2) + \psi( \tilde \bbx^s) =  \langle  S( \tilde \bbx^s), x^s_1+x^s_2 \rangle \nn \\
\iff &(x^s_1,...,x^s_N) \in \Gamma = \bigg\{\bbx \ \bigg| \ \sum_{j=1}^N f_j (x_j) = c(\bbx) \bigg\}.
\end{align*}
This completes the proof of the theorem.
\end{proof}

\begin{remark}\label{remark2}
The proof shows that Theorem \ref{main} holds for any Hilbert space $\cH$ and for all $N \ge 3$ as soon as it holds for $N=3$. However, the case $N=3$ necessitates the restriction $\cH = \R$ for the theorem to hold unless some regularity assumption is imposed on the conjugate convex functions. In view of Remark \ref{remark1}, this implies the set ${\cal I}$ is not dense in $\cH$ in general when $\cH$ is multi-dimensional. In the following section, we provide examples for illustration.
\end{remark}

\section{Failure on multidimensional domain: Examples on the plane}\label{sec4}
Recall the notion of $c$-cyclical monotonicity. Let $e_1 = (1,0)$, $e_2 = (0,1)$ denote the standard basis of $\R^2$. We now provide details for  Proposition \ref{counter1} in the following example.

\begin{example}\label{counter}
Let $\la > 0$, and $u =(1,0) \in \R^2$, $v = (\tfrac12, \tfrac{\sqrt{3}}2) \in \R^2$. There exists a $c$-conjugate triple $f,g,h \in {\cal A}(\R^2)$  for which $\Gamma = \Gamma_{\{f,g,h\}}$ satisfies
\be\label{SGamma}
S(\Gamma) = \R^2 \setminus {\rm int}\, {\mathbf H}_\la 
\ee
where ${\mathbf H}_\la$ is the convex hull of its six vertices $\pm 2\la u$, $\pm 2\la v$, and $\pm 2\la (v - u)$. Furthermore, $\Gamma$ is not maximally $c$-monotone. 

An example of such a $c$-conjugate triple is as follows:
\begin{align}\label{triple}
f(x) &= \begin{cases}
0 \, \text{ on }\,  I_1 := \{ x = au \ | \ a \in [-\la,\la]\}, \\
\infty  \, \text{ else,} \nn
\end{cases}
\\
g(y) &= \begin{cases}
0 \, \text{ on }\,  I_2 := \{ y = bv \ | \ b \in [-\la,\la]\}, \\
\infty  \, \text{ else,} \nn
\end{cases}
\\
h(z) &= \max \big( \la | u \cdot z + \la u \cdot v | + \la v \cdot z, \, \la | u \cdot z -\la u \cdot v | - \la v \cdot z \big). \nn
\end{align}
\end{example} 

\begin{proof}
{\bf Step 1: $h = (f \oplus g)^c$.} We firstly show $h = (f \oplus g)^c$. This follows from the following straightforward calculation:
\begin{align*}
(f \oplus g)^c (z) &= \sup_{x,y \in \R^2} x \cdot y + y \cdot z + z \cdot x - f(x) - g(y) \\
&= \sup_{|a| \le \la, |b| \le \la} (au + bv) \cdot z + ab (u \cdot v) \\
&= \sup_{|b| \le \la} \sup_{|a| \le \la} a(u \cdot z + b(u\cdot v)) + b(v \cdot z) \\
&=  \sup_{|b| \le \la} \la | u \cdot z + b(u\cdot v) | + b (v \cdot z) \\
&= \max \big( \la | u \cdot z + \la u \cdot v | + \la v \cdot z, \, \la | u \cdot z -\la u \cdot v | - \la v \cdot z \big)
\end{align*}
where the last equality is because the function $b \mapsto \la | u \cdot z + b(u\cdot v) | + b (v \cdot z) $ is convex and thus its maximum is attained at the boundary $b=\la$ or $b=-\la$.
\\

\noin{\bf Step 2: geometry of $h$ and $h^*$.} We investigate the geometry of $h$ and thereby derive $h^*$. Recall $u =e_1$, $v = \tfrac12 e_1 + \tfrac{\sqrt{3}}2 e_2$. Thus, for $z = (z_1,z_2) \in \R^2$,
\be
h(z) = \max \big( \la\big|z_1 + \tfrac12\la \big| + \tfrac12 \la z_1 + \tfrac{\sqrt3}2 \la z_2, \, \la\big|z_1 - \tfrac12 \la \big| - \tfrac12 \la z_1 - \tfrac{\sqrt3}2 \la z_2 \big). \nn
\ee
The solution of the equation
\be
\big|z_1 + \tfrac12 \la \big| + \tfrac12 z_1 + \tfrac{\sqrt3}2 z_2 = \big|z_1 - \tfrac12 \la \big| - \tfrac12 z_1 - \tfrac{\sqrt3}2 z_2 
\nn
\ee
is given by
\begin{align*}
z_2 &= \tfrac1{\sqrt3}\big( \big| z_1 - \tfrac12 \la \big| - \big| z_1 + \tfrac12 \la \big| - z_1 \big) \\
&= \begin{cases}
-\sqrt{3}z_1 \, \text{ if } \, z_1 \in [-\frac12 \la, \frac12 \la], \\
-\frac{1}{\sqrt{3}}z_1 -\frac{1}{\sqrt{3}}\la \, \text{ if } \, z_1 \in [ \frac12 \la, +\infty), \\
-\frac{1}{\sqrt{3}} z_1 + \frac{1}{\sqrt{3}}\la \, \text{ if } \, z_1 \in ( -\infty, -\frac12 \la].
\end{cases}
\end{align*}
This, and the presence of the terms $\big|z_1 + \frac12 \la \big|$ and $\big|z_1 - \frac12 \la \big|$ imply that $h$ is a piecewise affine convex function on the following four closed convex regions
\begin{align*}
R_1 &= \{ z_1 \ge \tfrac12 \la\} \cap \{ z_2 \le -\tfrac{1}{\sqrt{3}}z_1 -\tfrac{1}{\sqrt{3}} \la \},\\
R_2 &= \{ z_1 \le - \tfrac12 \la\} \cap \{ z_2 \ge -\tfrac{1}{\sqrt{3}} z_1 + \tfrac{1}{\sqrt{3}} \la \},\\
R_3 &= \{ z_2 \ge -\sqrt3 z_1\} \setminus ( {\rm int} R_1 \cup {\rm int} R_2), \\
R_4 &= \{ z_2 \le -\sqrt3 z_1\} \setminus ( {\rm int} R_1 \cup {\rm int} R_2),
\end{align*}
where int$R$ denotes the interior of $R$, and $h$ is defined on each region as
\be\label{h}
h(z) = \begin{cases}
\tfrac{1}{2} \la z_1 - \tfrac{\sqrt3}{2} \la z_2 - \tfrac12 \la^2 \, \text{ in } R_1,\\
-\tfrac{1}{2} \la z_1 + \tfrac{\sqrt3}{2} \la z_2 - \tfrac12 \la^2 \, \text{ in } R_2, \\
 \tfrac{3}{2} \la z_1 +  \tfrac{\sqrt3}{2}\la  z_2 + \tfrac12 \la^2 \, \text{ in } R_3, \\
  -\tfrac{3}{2} \la z_1 -\tfrac{\sqrt3}{2} \la z_2 + \tfrac12 \la^2  \, \text{ in } R_4.
\end{cases}
\ee
See Figure \ref{figure1}. In particular, $\min h = h(0) = \tfrac12 \la^2$. Let $w = v - u =  -\tfrac12 e_1 + \tfrac{\sqrt3}{2} e_2$, and let $D_\la \subset \R^2$ denote the convex hull of its  four vertices $\pm \la w$, $\pm \la (u+v)$. Observe \eqref{h} readily implies that its convex conjugate is the following: 
\be\label{h*}
h^*(z^*) = \begin{cases}
\la^2( t-\tfrac{1}{2} ) \, \text{ if } \, z^* = \pm t \la w + s \la (u +v) \text{ for }  t \in [0,1], s \in [t-1, 1-t]
\\
+\infty \, \text{ else, i.e., if } \, z^* \notin D_\la.
\end{cases}
\ee

\begin{figure}
    \begin{subfigure}[b]{0.49\columnwidth}
        \includegraphics[width=\textwidth]{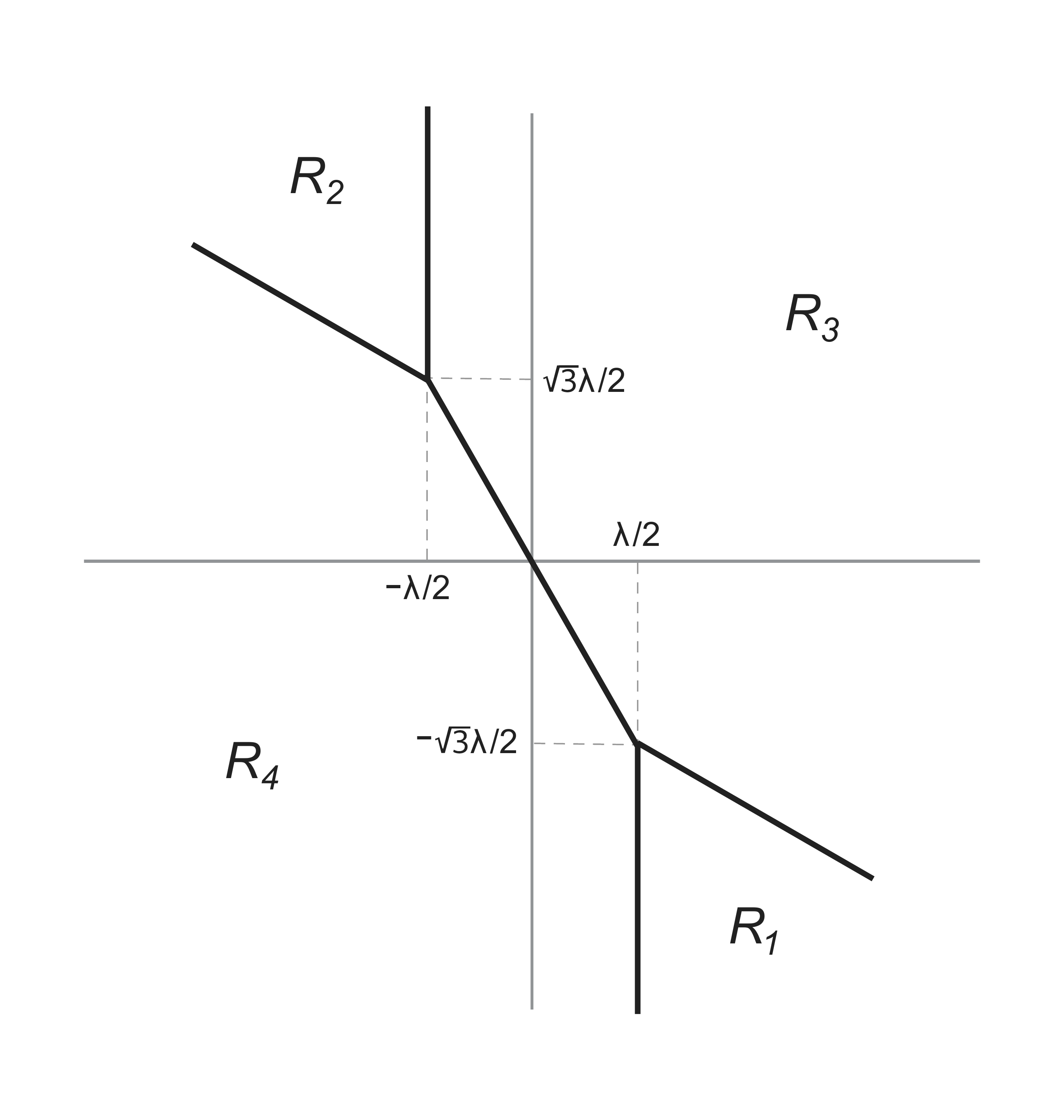}
      \caption{$h$ is piecewise affine and convex.}
    \end{subfigure}
    \hfill
    \begin{subfigure}[b]{0.49\columnwidth}
        \includegraphics[width=\textwidth]{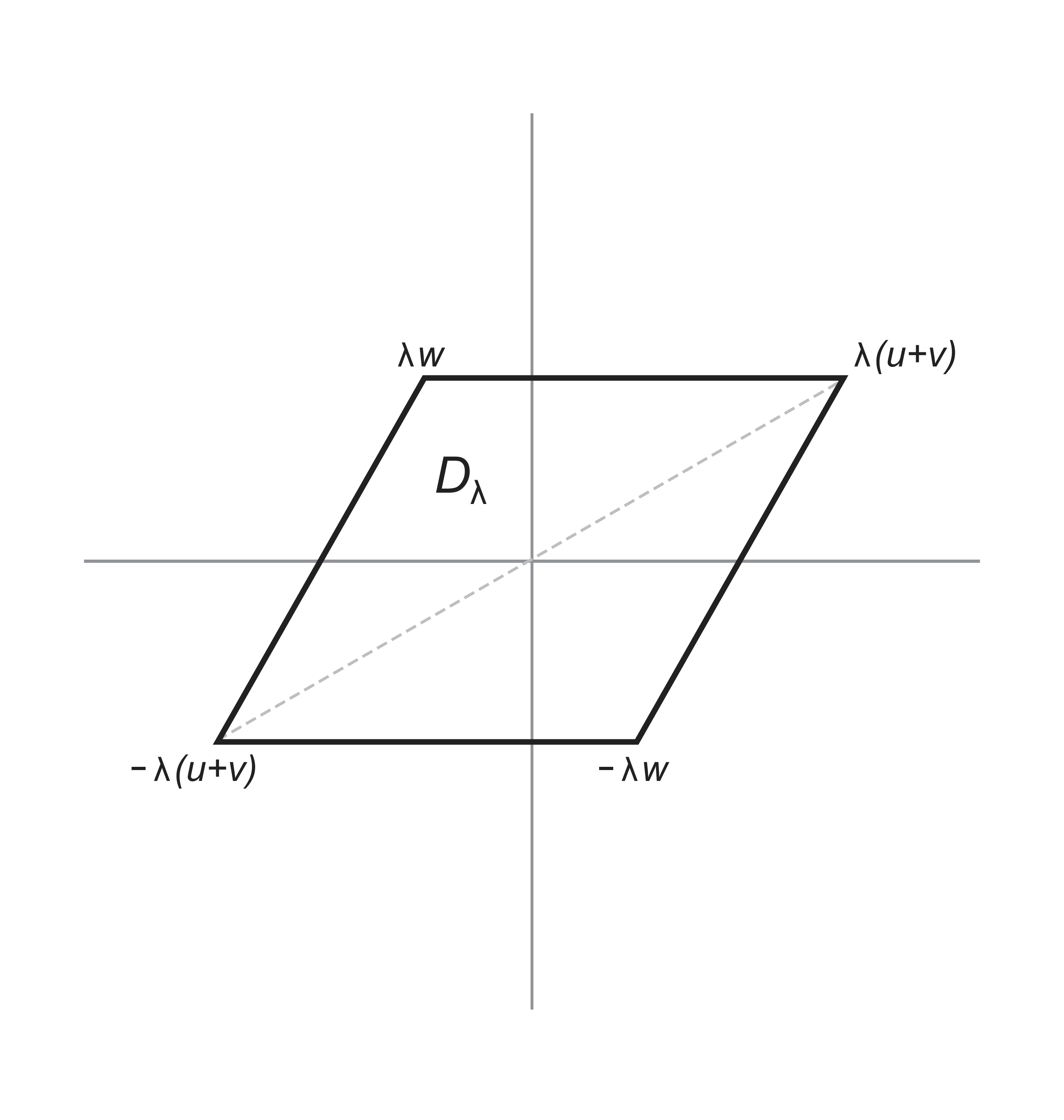}
      \caption{In $D_\la$, $h^*$ is constant along each line segment parallel to the dashed line.}
    \end{subfigure}
    \caption{Geometry of $h$ and $h^*$}
  \label{figure1}
\end{figure}

\noin{\bf Step 3: $g = (h \oplus f)^c$ and $f = (h \oplus g)^c$.} Now we show $g = (h \oplus f)^c$. By symmetry, $f = (h \oplus g)^c$ will then follow verbatim and we omit. We compute
\begin{align}
(h \oplus f)^c (y) &= \sup_{x,z \in \R^2} x \cdot y + y \cdot z + z \cdot x - f(x) - h(z) \nn \\
&= \sup_{z \in \R^2} \sup_{a \in [-\la,\la]} a(y_1+z_1) + y \cdot z -h(z) \nn \\
&= \sup_{z \in \R^2} y \cdot z - \big( h(z) - \la |z_1 + y_1| \big) \nn \\
&= (\xi_{y_1})^* (y) \nn
\end{align}
where $y= (y_1,y_2)$, $z=(z_1,z_2)$ and $\xi_{y_1} (z) := h(z) - \la |z_1 + y_1|$. We need to compute $(\xi_{y_1})^*$, the convex conjugate of $\xi_{y_1}$ for each $y_1$, and evaluate at $y$. First, since the coefficient of $z_2$ in \eqref{h} is within $[-\tfrac{\sqrt3 \la}{2}, \tfrac{\sqrt3 \la}{2}]$, we readily get
\be
(\xi_{y_1})^* (y) = +\infty \, \text{ if } \, |y_2| > \tfrac{\sqrt3 \la}{2}.\nn
\ee
As a result, we will assume $y_2 \in [-\tfrac{\sqrt{3} \la}2, \tfrac{\sqrt{3} \la}2]$ from now on. We claim that $(\xi_{y_1})^* (y) = +\infty$ if $y$ is not parallel to $v$. To see this, we will show that $\xi_{y_1}$ is constant on the half-infinite lines $L_1 := R_1 \cap R_3$ and $L_2 := R_2 \cap R_4$ respectively. Notice the claim then follows from this constancy and the fact that $v$ is perpendicular to the lines $L_1, L_2$. Now it is clear that $\xi_{y_1}$ is constant on $L_1$, since for $ z \in L_1$, we have $z_2 = -\tfrac1{\sqrt3} z_1  -\tfrac1{\sqrt3}\la$, yielding $h(z) =\tfrac{\la}{2} z_1 - \tfrac{\sqrt3 \la}{2} z_2 - \tfrac{\la^2}2 = \la z_1$. Similarly, $h(z) = -\la z_1$ on $L_2$, proving the claim. 

The remaining case for calculating $(\xi_{y_1})^* (y)$ is when $y$ is parallel to $v$, that is, $y_2 = \sqrt3 y_1$ and $y_1 \in [-\tfrac{\la}2, \tfrac{\la}2]$. We need to show $(\xi_{y_1})^* (y) = 0$, which will follow if we can show that the function $l_y (z) := y \cdot z$ satisfies $l_y \le \xi_{y_1}$ in $\R^2$ with the inequality being saturated at some point in $\R^2$. In fact, we claim
\begin{align}
&l_y = \xi_{y_1} \text{ on } L_1 \cup L_2, \, \text{ and } \, l_y \le \xi_{y_1} \text{ on } L_3 \cup L_4 \cup L_5 \cup L_6, \text{ where}  \nn \\
&L_3 = R_1 \cap R_4, L_4 = R_2 \cap R_3, L_5 = R_3 \cap R_4, L_6 = \{z \ | \ z_1 = -y_1\}. \nn
\end{align}
Notice that since $\xi_{y_1}$ is piecewise affine, the claim implies $l_y \le \xi_{y_1}$ as desired.

\begin{figure}
    \centering 
    \begin{subfigure}[H]{0.49\columnwidth}
        \includegraphics[width=\textwidth]{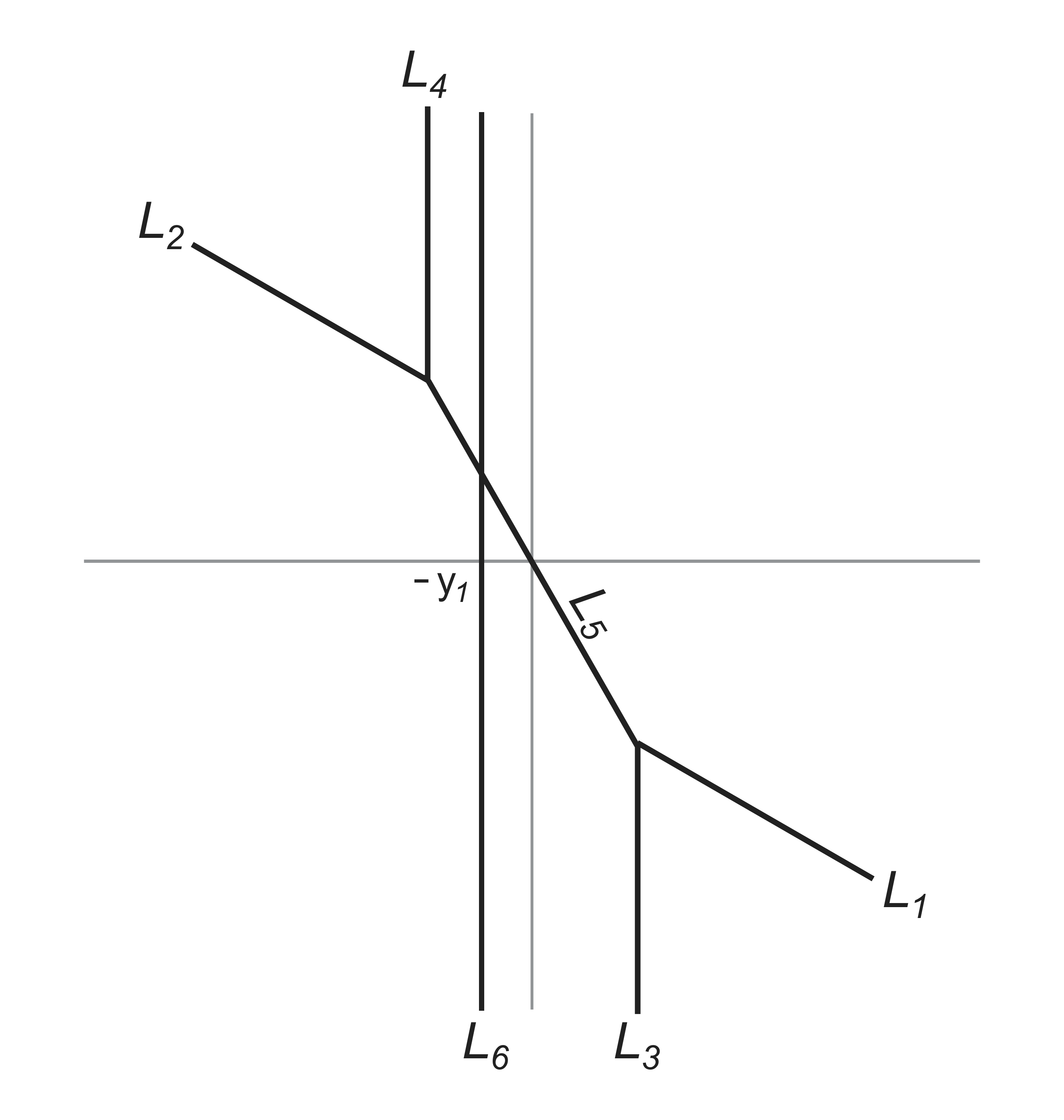}
    \end{subfigure}
    \caption{$l_y \le \xi_{y_1}$ along each line segment $L_i$, thus in $\R^2$}
  \label{figure2}
\end{figure}

To see $l_y = \xi_{y_1}$ on $L_1$, since $l_y$ is constant on $L_1$, it is enough to show $l_y (-\la w) = \xi_{y_1}(-\la w)$. But $\xi_{y_1}(-\la w) = h(- \la w) - \la | y_1 - \la w_1 | =\tfrac{\la^2}2 - \la(y_1 + \tfrac{\la}2) =  - \la y_1 = \la(\tfrac12, - \tfrac{\sqrt3}{2}) \cdot (y_1, \sqrt3 y_1) = -\la w \cdot y$, as desired. Similarly, $l_y = \xi_{y_1}$ on $L_2$. Then $l_y \le \xi_{y_1}$ on $L_4$ is immediate from the fact that $L_4$ is vertical, and the coefficient of $z_2$ in \eqref{h} is $ \tfrac{\sqrt3 \la}{2}$ which is no less than $y_2$, meaning that $\xi_{y_1}$ grows no slower than $l_y$ along $L_4$. Similarly $l_y \le \xi_{y_1}$ on $L_3$. Then $l_y \le \xi_{y_1}$ on $L_5$ follows from the fact that $\xi_{y_1}$ is concave on $L_5$ with $l_y = \xi_{y_1}$ at the boundary points of $L_5$, i.e., at $\la w$ and $-\la w$. From this, the dominance $l_y \le \xi_{y_1}$ on $L_6$ can also be seen from the growth of $h$ in $z_2$ direction, along with the fact $l_y \le \xi_{y_1}$ at the point $L_5 \cap L_6$ just established. This proves the conjugacy $g = (h \oplus f)^c$.
\\

\noin{\bf Step 4: verification of $\Gamma$.} We want to characterize $\Gamma$ induced by the $c$-conjugate triple $(f,g,h)$. Recall that $f(x)$ and $g(y)$ are finite only for $x = au, y = bv$ with $ a,b  \in [-\la,\la]$, so we assume this henceforth. The inequality
\[
f(x) + g(y) + h(z) \ge x \cdot y + y \cdot z + z \cdot x
\]
then becomes 
\be\label{ineq1}
h(z) \ge z\cdot (x+y) + x \cdot y.
\ee
Notice this inequality can only become equality for $(x,y)$ that satisfies:
\be\label{eq1}
h^*(x+y) = - x \cdot y.
\ee
That is, only for those $(x,y)$ satisfying \eqref{eq1} can there be $z$ satisfying equality in \eqref{ineq1}, i.e., $(x,y,z) \in \Gamma$. Thus we want to solve \eqref{eq1}. Recall \eqref{h*} that $h^* (z^*)$ is constant (and equal to $\la^2(t - \tfrac12)$) for $z^* = \pm t\la w + s\la(u +v)$, $t \in [0,1]$, $s \in [t-1, 1-t]$, where $w = v - u$. In view of \eqref{eq1}, let us assume  $x+y = t \la(v-u) + s \la(u+v)$. Notice this implies $x = (s-t)\la u$, $y = (s+t)\la v$, since $x,y$ are parallel to $u,v$ respectively. In this case, \eqref{eq1} yields
\begin{align*}
&\la^2(t - \tfrac12) = - \tfrac12 \la^2 (s^2 - t^2) \\
&\iff  s^2 = (t-1)^2 \\
& \iff s = t-1 \, \text{ or }\, 1-t.
\end{align*}
In view of \eqref{h*} and the definition of the region $D_\la$, this precisely means 
\be\label{x+y}
\{ x+y  \ | \ h^*(x+y) = - x \cdot y, \, x = au, \, y = bv, \, -\la \le a,b \le \la \} = \p D_\la
\ee
where $\p D_\la$ is the boundary of $D_\la$. Given $(x,y)$ such that $x+y \in \p D_\la$, finding $\Gamma_{x,y} := \{z \ | \ (x,y,z) \in \Gamma\}$ yielding equality in \eqref{ineq1} is now straightforward from the description of $h$ in \eqref{h}. For example, for $x= \la u, y = \la v$, it is clear that $\Gamma_{x,y}$ is precisely $R_3$, since the slope of $h$ is $x+y = \la (u+v)$ on $R_3$. Similarly, for $x = - \la u, y = \la v$, we find $\Gamma_{x,y} =R_2$. For $x = t\la u, y= \la v$ with $t \in (-1,1)$, $\Gamma_{x,y} = R_2 \cap R_3 = L_4$. In this way, we can describe $\Gamma$ completely as follows: 
\begin{align}\label{Gammadetail}
\Gamma_{x,y} = \begin{cases}
R_3 \q \q \ \ \text{ if } \ x= \la u, y = \la v,
\\
R_2 \cap R_3  \ \text{ if } \ x= t\la u, y = \la v \text{ and } t \in (-1,1),
\\
R_2  \q \q \ \ \text{ if } \ x= -\la u, y = \la v,
\\
R_2 \cap R_4 \ \text{ if } \ x= -\la u, y = t\la v \text{ and } t \in (-1,1),
\\
R_4 \q \q \ \ \text{ if } \ x= -\la u, y = -\la v,
\\
R_4 \cap R_1 \ \text{ if } \ x= t\la u, y = -\la v \text{ and } t \in (-1,1),
\\
R_1 \q \q \ \ \text{ if } \ x= \la u, y = -\la v,
\\
R_1 \cap R_3 \ \text{ if } \ x= \la u, y = t\la v \text{ and } t \in (-1,1).
\end{cases}
\end{align}
Gathering $x+y + \Gamma_{x,y}$ for every $x+y \in \p D$, we conclude that $S(\Gamma)$ is  precisely as described in \eqref{SGamma}. 
\\

\noin{\bf Step 5: $\Gamma$ is not maximally $c$-monotone.} Note that $(0,0,0) \in (\R^2)^3$ is not in $\Gamma$. We will show that $\tilde \Gamma := \Gamma \cup \{(0,0,0)\}$ is still $c$-monotone, thereby showing that $\Gamma$ is not maximally $c$-monotone. In view of Definition \ref{monotonicity definitions} and the fact that $\Gamma$ is $c$-monotone, it is enough to prove the following for any $(x,y,z) \in \Gamma$:
\begin{align}
c(x,y,z) + c(0,0,0) \ge c(0,y,z) + c(x,0,0), \nn \\
c(x,y,z) + c(0,0,0) \ge c(x,0,z) + c(0,y,0), \nn \\
c(x,y,z) + c(0,0,0) \ge c(x,y,0) + c(0,0,z). \nn 
\end{align}
That is, we need to show
\begin{align}\label{claim8}
z \cdot (x+y ) \ge 0, \ x \cdot (y+z) \ge 0, \ y \cdot (z+x) \ge 0  \ \text{ for any }  (x,y,z) \in \Gamma.
\end{align}
In each case of \eqref{Gammadetail}, we can directly check \eqref{claim8}. In the first case where $x = \la u$, $y = \la v$ and $z \in R_3$, the inequality $z \cdot (x+y) \ge 0$ is obvious from the direction of $u+v$ and the definition of $R_3$. By observing that $y+ R_3$ has nonnegative first components, the inequality $x \cdot (y+z) \ge 0$ is easily seen. Likewise, $y \cdot (z+x) \ge 0$ is checked. We can also easily check that other cases in \eqref{Gammadetail} where $\Gamma_{x,y} = R_2$, $\Gamma_{x,y} = R_4$, and $\Gamma_{x,y} = R_1$ satisfy \eqref{claim8} as well. The remaining cases in \eqref{Gammadetail} then satisfy \eqref{claim8} through interpolation, verifying that $\tilde \Gamma$ is $c$-monotone.
\end{proof}
\begin{remark}
The author conjectures that $\tilde \Gamma$ is also $c$-cyclically monotone, and thus $\Gamma$ is not maximally $c$-cyclically monotone either. However, verifying the $c$-cyclically monotonicity of $\tilde \Gamma$ appears to be nontrivial (despite the fact that we have a complete description of $\Gamma$!), and we leave it as an open question.
\end{remark}

The oblique direction of $u,v$ (i.e., $0 < u \cdot v < 1$), as well as finite-length but nonzero support of $f$ and $g$ (i.e., $ 0 < \la < \infty$), appear to be essential in constructing a counterexample presented in Example \ref{counter}. In the following examples, we illustrate these observations.

\begin{example}[Perpendicular $u,v$ yields $S(\Gamma) = \R^2$.]\label{ex1:perp}
Let  $u = e_1$, $v = e_2$ and $\la \in [0, \infty]$ in Example \ref{counter} (note that we allow $\la = 0$ or $\infty$), so that
\begin{align}\label{triple}
f(x) &= \begin{cases}
0 \, \text{ on }\,  I_1= \{ x = au \ | \ a \in [-\la,\la] \}, \\
\infty  \, \text{ else,} \nn
\end{cases}
\\
g(y) &= \begin{cases}
0 \, \text{ on }\,  I_2 = \{ y = bv \ | \ b \in [-\la,\la] \}, \\
\infty  \, \text{ else,} \nn
\end{cases}
\\
h(z) &= (f \oplus g)^c (z) \nn \\
&= \sup_{a,b \in [-\la,\la]} (au+bv) \cdot z \nn \\
&= \begin{cases}
\la (|z_1| + |z_2|) \, \text{ if } \, \la \in [0, \infty), \nn \\
0 \, \text{ if } \, z=0, \  \infty \text{ else, } \, \text{ if } \la = \infty.
\end{cases}
\end{align}
$(f,g,h)$ is a $c$-conjugate triple as easily verified. If $\la = 0$, then $(0,0,z) \in \Gamma$ for any $z \in \R^2$, yielding $S(\Gamma) = \R^2$. If $\la = \infty$, then $(au,bv,0) \in \Gamma$ for any $a,b \in \R$, again yielding $S(\Gamma) = \R^2$. Thus we henceforth assume $\la \in (0, \infty)$. We have
\be
h^*(z^*) = \begin{cases}
0 \, \text{ if } \, z^* = (z^*_1, z^*_2), \, -\la \le z^*_1, z^*_2 \le \la,\\
\infty \, \text{ else.}
\end{cases} \nn
\ee
Then the pair of $(x,y) \in I_1 \times I_2$ satisfying \eqref{eq1} is equal to $I_1 \times I_2$, since $x \cdot y = 0$. Thus the sum set, i.e., the set of $x+y$, becomes
\be
\{ x+y  \ | \ h^*(x+y) = - x \cdot y, \ x \in I_1, y \in I_2 \} =: D
\ee
where $D = I_1 + I_2$ is the convex hull of its four vertices $\pm \la(u+v), \pm \la(v-u)$. Notice the difference between \eqref{x+y} in Step 4 of Example \ref{counter}, where the sum set is equal to the boundary of $D$, and not all of $D$.

Now for $(x,y)$ with $x+y \in {\rm int} D$, clearly $\Gamma_{x,y} = \{z \ | \ (x,y,z) \in \Gamma\} = \{0\}$. For each $(x,y)$ such that $x+y \in \p D$, verifying $\Gamma_{x,y}$ is again straightforward from the definition of $h$. For example, it is easy to see that $\Gamma_{\la u, \la v} =\{ z = (z_1, z_2) \in \R^2 \ | \ z_1 \ge 0, z_2 \ge 0\}$, $\Gamma_{-\la u, \la v} =\{ z = (z_1, z_2) \in \R^2 \ | \ z_1 \le 0, z_2 \ge 0\}$, and $\Gamma_{t\la u, \la v} =\{ z = (z_1, z_2) \in \R^2 \ | \ z_1 = 0, z_2 \ge 0\}$ for any $t \in (-1,1)$, and so on. Gathering $x+y + \Gamma_{x,y}$ for every $x+y \in D$, we verify $S(\Gamma) = \R^2$. We thus conclude that the triple $(f,g,h)$ does not provide a counterexample to Question \ref{question1} when $u$ and $v$ are perpendicular.
\end{example}

\begin{example}[Infinite support of $f, g$ with $u \cdot v \neq 0$ yields $h \equiv +\infty$.]\label{ex2:infinite}
Let $u,v$ be non-perpendicular unit vectors in $\R^2$, i.e., assume $u \cdot v \in (0,1]$. Define
\begin{align*}
f(x) &= \begin{cases}
0 \, \text{ on }\,  I_1= \{ x = au \ | \ a \in \R \}, \\
\infty  \, \text{ else,} \nn
\end{cases}
\\
g(y) &= \begin{cases}
0 \, \text{ on }\,  I_2 = \{ y = bv \ | \ b \in \R \}, \\
\infty  \, \text{ else.} \nn
\end{cases}
\end{align*}
In this case, we claim that $h := (f \oplus g)^c$ is not proper but $h \equiv +\infty$. Because of this, the triple $(f,g,h)$ does not provide a counterexample to Question \ref{question1}.

The claim is straightforward from the following calculation:
\begin{align*}
(f \oplus g)^c (z) &= \sup_{x,y \in \R^2} x \cdot y + y \cdot z + z \cdot x - f(x) - g(y) \\
&= \sup_{a \in \R, b \in \R} (au + bv) \cdot z + ab (u \cdot v) \\
& \ge \sup_{a \in \R} a(u + v) \cdot z + a^2 (u \cdot v) \\
& = \infty
\end{align*}
for any $z$, since $u \cdot v > 0$.
\end{example}

\section{Dual perspective of cyclical conjugation and maximality}\label{sec5}
In this section, we illustrate cyclical involutivity and maximal monotonicity in the $N=3$ case in the context of duality. 
For $f,g,h \in {\cal A}(\cH)$, we recall the following inequality
\be\label{ineq4}
f(x) + g(y) + h(z) \ge \langle x,y \rangle +  \langle y,z \rangle +  \langle z,x \rangle, \ x,y,z \in \cH
\ee
is equivalent to
\be\label{ineq5}
F(x) + G(y) \ge H(x+y)
\ee
where $F = f +q$, $G = g + q$, $H = h^* + q$, which in turn yields the equivalence
\begin{align}
&f = (g \oplus h)^c \, \text{ and } \, g = (h \oplus f)^c \label{fgconjugate} \\
&\iff F(x) = \sup_{y \in \cH} H(x+y) - G(y) \, \text{ and } \, G(y) = \sup_{x \in \cH} H(x+y) - F(x) \nn \\
&\iff F = (H^*-G^*)^* \, \text{ and } \, G = (H^*-F^*)^* \ \text{ by Lemma \ref{switch}} \nn \\
&\iff U = (W - V)^{**}  \, \text{ and } \, V = (W - U)^{**} \label{dualproblem}
\end{align}
where $U = F^*$, $V = G^*$, $W = H^*$. This naturally leads us to  consider the following class of convex functions (see \cite[Propositions 12.30, 14.2]{BC2017})
\begin{align}
{\cal A}_1(\cH) := &\, \{F \in {\cal A}(\cH) \ | \ F \text{ is $1$-strongly convex}\}, \nn \\
 {\cal A}_{1}^*(\cH) := &\, \{U \in {\cal A}(\cH) \ | \ U = F^* \text{ for some } F \in {\cal A}_1(\cH) \} \nn \\
 = &\, \{U \in {\cal A}(\cH) \ | \ U = f \square q \text{ for some } f \in {\cal A}(\cH) \} \nn \\
 =&\, \{U \in {\cal A}(\cH) \ | \ q - U  \text{ is continuous and convex} \}. \nn
\end{align}
We see that, given \eqref{fgconjugate}, the question of $3$-cyclical involutivity, that is, whether $h$ is equal to $(f \oplus g)^c$ or not, can be recast as the following dual problem: 
\vspace{1mm}

\noin Is $W := h \square q$ the smallest (pointwise in $\cH$) function in ${\cal A}_{1}^*(\cH)$ satisfying \eqref{dualproblem}?
\\

Note that $W := U+V$ is obviously the smallest convex function satisfying \eqref{dualproblem}, but $U+V$ may not belong to the class ${\cal A}_{1}^*(\cH)$ in general. This makes the question of finding the smallest $W$ in ${\cal A}_{1}^*(\cH)$ nontrivial. Indeed, we show below how the cyclical involution can fail only under the assumption \eqref{fgconjugate}.
\\

Now we turn to the maximality question \eqref{maximality} in the context of duality. Assume $f,g,h \in {\cal A}(\cH)$ satisfies \eqref{ineq4}. The equivalence between \eqref{t:split_max_c-mono-v} and \eqref{t:split_max_c-mono-vii} in Theorem \ref{t:split_max_c-mono} is then recast, via Moreau decomposition, as follows:
\be\label{equiv1}
\Gamma = \Gamma_{\{f,g,h\}} \text{ satisfies \eqref{maximality} if and only if } W = U+V.
\ee
Now assume $f,g \in {\cal A}(\cH)$, and set $h := (f \oplus g)^c$. Observe that \eqref{equiv1} then yields the following equivalence:
\be\label{equiv2}
\Gamma = \Gamma_{\{f,g,h\}} \text{ satisfies \eqref{maximality} if and only if } U+V \in {\cal A}_{1}^*(\cH).
\ee
This is because if $W:= U+V$ is in ${\cal A}_{1}^*(\cH)$, then $W^* - q$ is convex, thus the corresponding triple $f := U^* - q$, $g := V^* - q$ and $h := (W^* - q)^*$ satisfies \eqref{maximality} and $c$-conjugate, by the equivalence between \eqref{t:split_max_c-mono-v} and \eqref{t:split_max_c-mono-vii} in Theorem \ref{t:split_max_c-mono}.

Let us explore the above discussion through the examples from the previous section. In Example \ref{counter}, the triple $(f,g,h)$ is $c$-conjugate, which yields $W=h \square q$ is indeed the smallest element in ${\cal A}_{1}^*(\cH)$ satisfying \eqref{dualproblem}. However, $W(0) = -\min W^*> 0 = U(0)+V(0)$, causing the equation $W = U+V$ to fail. We deduce $\Gamma$ does not satisfy \eqref{maximality}, without having to precisely compute $S(\Gamma)$.

On the other hand, in Example \ref{ex1:perp}, we have  
\begin{align*}
U(x) =(f+q)^*(x_1,x_2) = \begin{cases}
\frac12 |x_1|^2 \, \text{ for }\, x_1 \in [-\la,\la], \\
\la x_1 - \frac 12 \la^2 \, \text{ for }\, x_1 \ge \la, \\
-\la x_1 - \frac 12 \la^2 \, \text{ for }\, x_1 \le -\la,
\end{cases}
\end{align*}
and $V(x_1,x_2) = U(x_2,x_1)$. It is easy to see that $U+V$ belongs to ${\cal A}_{1}^*(\cH)$, and hence by \eqref{equiv2}, we conclude that the triple $(f,g,h)$ does satisfy \eqref{maximality}.

Finally, we conclude this paper by detailing Proposition \ref{noinvolutive}.

\begin{example}\label{noinvolutive1}
We will construct $f,g,h \in {\cal A}(\R^2)$ such that $f = (g \oplus h)^c$, $g = (h \oplus f)^c$, but $h \neq (f \oplus g)^c$. We denote $x = (x_1,x_2) \in \R^2$, $y=(y_1,y_2) \in \R^2$. Let $a \lor b := \max \{a,b\}$. Recall $\{e_1, e_2\}$ denote the standard basis of $\R^2$. Define
\begin{align}
H_1(x) &=|x_1| \lor |x_2| + q(x), \nn \\
G_0(y) &= \begin{cases}
H_1(y) \ \text{ if } \ y = e_1 \text{ or }\, y=-e_1,\\
+\infty \ \text{ else,}
\end{cases} \nn
\\
F_1(x) &= \sup_{y \in \R^2} H_1(x+y) - G_0 (y), \nn \\
G_1(y) &= \sup_{x \in \R^2} H_1(x+y) - F_1 (x). \nn
\end{align}
Then $F_1, G_1$ are $H_1$-conjugate. And by definition of $F_1$, we have 
\be
F_1(0) = \min F_1 = 0, \text{ and } \  \frac{\p F_1}{\p x_2} (0) = 0, \nn
\ee
that is, $F_1$ is smooth at its minimum (origin) in the $x_2$-direction. Next, notice $G_1 = G_0$ at $y = \pm e_1$, and by symmetry of the construction, we have $G_1 (0) = \min G_1$. We claim $G_1(0) > 0$. This can be easily seen from the inequality 
\be
F_1 (x) \ge H_1(x) - G_1(0) \ \text{ for all } \ x \in \R^2, \nn
\ee
and the fact that $F_1(0) = H_1(0) = 0$, and at the origin $x=0$, $F_1$ is smooth in $x_2$-direction but $H_1$ is singular. This proves the claim. Let $m := G_1(0) > 0$, and redefine
\be
H = H_1 - m, \q F = F_1, \q G = G_1 - m, \nn
\ee
so that we have 
\begin{align}
&F(0) = \min F =  G(0) = \min G = 0, \ H(0) = \min H  < 0, \text{ and } \label{zeromin}\\
& F,G \text{ are $H$-conjugate.} \label{conjugateok}
\end{align}
Now we define $f,g,h \in {\cal A}(\R^2)$ by 
$f = F - q$, $g = G - q$, and $h = (H - q)^*$. Then as discussed before, \eqref{conjugateok} implies that $f = (g \oplus h)^c$ and $g = (h \oplus f)^c$. We need to show $h \neq (f \oplus g)^c$. To show this, rather than finding precise formulas for $h$ and $(f \oplus g)^c$, we derive a contradiction. Recall that the equality $h = (f \oplus g)^c$ implies $ h^* + q$ is the largest $1$-strongly convex lower-semicontinuous function among all $H \in {\cal A}_1(\R^2)$ satisfying \eqref{ineq5}. We claim that this is false with $h = (H - q)^*$. 

Set $\tilde H := H \lor 0$. Notice that \eqref{zeromin} then readily implies $F,G$ are $\tilde H$-conjugate as well. But $\tilde H$ is not $1$-strongly convex yet. Note that for any $K \in {\cal A}(\R^2)$ satisfying $H \le K \le \tilde H$, we have that $F,G$ are $K$-conjugate. Hence the claim will follow if we can find a $1$-strongly convex function $K$ satisfying:
\be\label{lastclaim}
H \le K \le \tilde H, \text{ and }\, H(0) < K(0).
\ee
But the definitions $H(x) = |x_1| \lor |x_2| + q(x) - m$ and $ \tilde H = H \lor 0$ readily imply that for all sufficiently small $\ep > 0 $, the following function
\be
K_\ep  := H \lor ( q - m + \ep) \nn
\ee
is a $1$-strongly convex function satisfying \eqref{lastclaim}. This yields $h \neq (f \oplus g)^c$.
\end{example}



\bibliographystyle{siamplain}
\bibliography{references}

\end{document}